\RequirePackage{fix-cm}
\documentclass[smallextended]{svjour3}
\smartqed
\usepackage{graphicx}
\graphicspath{ {images/} }

\usepackage{amsmath}
\usepackage{amssymb}
\usepackage{bbm}
\usepackage{xcolor}
\usepackage{url}
\usepackage{adjustbox}
\usepackage{chemmacros}
\usepackage{multirow}
\usepackage{tikz}
\usepackage{hyperref}
\usepackage{booktabs}

\usepackage{algorithm}
\usepackage{algpseudocode}

\makeatletter
\def\blfootnote{\gdef\@thefnmark{}\@footnotetext}
\makeatother

\title{Variable aggregation for nonlinear optimization problems}
\author{Sakshi Naik \and Lorenz Biegler \and Russell Bent \and Robert Parker}

\institute{
  SN and LB \at
  Department of Chemical Engineering, Carnegie Mellon University, Pittsburgh, PA, USA
  \and
  RB and RP \at
  Applied Math and Plasma Physics Group, Los Alamos National Laboratory, Los Alamos, NM, USA
}

\date{Received: date / Accepted: date}

\newcommand{\rev}[1]{{#1}}

\begin{document}
\maketitle

\begin{abstract}
Variable aggregation has been largely studied as an important pre-solve algorithm for
optimization of linear and mixed-integer programs. Although some nonlinear solvers
and algebraic modeling languages implement variable aggregation as a pre-solve reduction, the
impact it can have on constrained nonlinear programs is unexplored. In this work, we
formalize variable aggregation as a pre-solve algorithm to develop reduced-space formulations
of nonlinear programs. 
A novel approximate maximum variable aggregation strategy is developed to aggregate
as many variables as possible. Furthermore, aggregation strategies that preserve the
problem structure are compared against approximate maximum aggregation. Our results
show that variable aggregation can generally help to improve the convergence reliability
of nonlinear programs. It can also help in reducing total solve time. However, Hessian
evaluation can become a bottleneck if aggregation significantly increases the number
of variables appearing nonlinearly in many constraints.
\keywords{Nonlinear optimization \and Presolve \and Graph theory}
\end{abstract}

\section{Introduction}
Pre-solve reductions have been important to develop linear and
mixed-integer optimization solvers into generic tools that can solve
real-world problems in practical settings \cite{bixby2007}.
One such pre-solve reduction is to substitute an expression for a variable
that is defined by an equality constraint.
For example, if an optimization problem contains the constraints
\[
  \begin{array}{c}
    x^2 + y^2 \leq 1 \\
    y = 2z + 1 \\
  \end{array}
\]
then the expression $2z+1$ may be substituted for $y$ in the first constraint
to reduce the number of variables and constraints in the optimization problem.
We refer to this substitution as variable {\it aggregation} and the resulting
optimization problem as a {\it reduced-space} formulation.
In a summary of Gurobi's \cite{gurobimanual} pre-solve methods for mixed-integer
linear programming, Achterberg et al. \cite{achterberg2020} show that variable
aggregation can be a source of significant speed-up, but also mention that it
can be a source of numerical errors if not implemented carefully.
In the context of nonlinear optimization, Bongartz \cite{bongartzthesis}
tests manually-chosen variable aggregations and remarks that ``the reduced-space
formulations reduced computational times and increased success rates of local solvers''
with the caveat that ``the smallest possible formulation did not necessarily result in
the best computational performance'' (see Sections 3.4.3.3 and 3.5 of \cite{bongartzthesis}).

For nonlinear local optimization, a generalization of variable aggregation
is to remove a subset of variables and equality constraints that have a nonsingular
Jacobian using the implicit function theorem. Parker et al. \cite{parker2022}
and Bugosen et al. \cite{bugosen2024gibbs} show
that implicit function reformulations can improve reliability of interior point methods
for nonlinear optimization, although these implicit function reformulations must
be chosen by an expert user.
While it is not easy to automatically choose general subsets of equations and variables
that always have nonsingular Jacobian matrices, it is relatively straightforward to
identify variables and equations that can be eliminated by explicit variable aggregation.
Here, we propose {\it automatic} variable aggregation in nonlinear optimization models as a
method to improve convergence reliability of an interior point solver.

Several existing software packages implement automatic
variable aggregation for nonlinear optimization models.
The AMPL algebraic modeling language \cite{ampl} scans constraints to identify
those written in the form ``$y = f(x,\dots)$''.
If $y$ is unbounded and has not already been aggregated,
the expression $f(x,\dots)$ will be substituted for $y$ everywhere it appears.
(This behavior is controlled by the \texttt{substout} option.
See \cite{fourer1994} and Chapters 14 and 18.2 of \cite{ampl} for details.)
The AIMMS modeling language \cite{aimmsmanual} and Pyomo modeling environment
\cite{pyomobook,pyomopaper} also implement presolve options for aggregating
variables, but only use equality constraints with at most two variables for
these aggregations.
Variable aggregation may also be performed within nonlinear optimization solvers.
The Knitro \cite{knitro} and CONOPT \cite{conopt} solvers implement options
for variable aggregation \cite{knitromanual,conoptgamsmanual}, but it is not clear
what criteria these commercial solvers use to select variables and constraints
to eliminate.

Despite the many variable aggregation implementations present in open-source and
commercial software for nonlinear optimization, the impacts of these pre-solve
strategies are not well-documented.
We fill this gap by analyzing several aggregation methods inspired by those implemented
in AMPL, AIMMS, and Pyomo. We analyze structural properties, solve time, and
convergence reliability of optimization models before and after aggregation.
Specifically, our contributions are:
\begin{enumerate}
  \item We formalize variable aggregation for nonlinear optimization problems
    and present an easily checkable sufficient condition for a subset of
    variables and equality constraints to admit a valid explicit aggregation.
  \item We present a new approximate-maximum algorithm for variable aggregation.
  \item We demonstrate that conservative variable aggregation approaches
    lead to modest improvements in solve time, while aggressive variable aggregation
    methods increase solve time by creating expensive-to-evaluate Hessian matrices.
  \item We demonstrate the convergence reliability benefits of variable aggregation
    on several parameterized nonlinear optimization test problems by sampling
    121 parameter values per problem and attempting to solve each instance.
    For each aggregation method, we attempt to solve each problem instance.
    Six different aggregation methods (and the original problem)
    are compared by the number of instances solved for each test problem.
\end{enumerate}
Improved convergence reliability of nonlinear optimization problems when aggregating
variables is significant as even state-of-the-art nonlinear optimization solvers
are not perfectly reliable. That is, convergence (within a pre-specified time or iteration limit)
can be sensitive to initialization, scaling factors, and parameter values.
Despite mature convergence theory for globalized optimization algorithms
(e.g., \cite{waechter2005global}) the reasons
for algorithms' sensitivities to these values is not well-understood.
By documenting the effect of variable aggregation on convergence reliability, we
build an empirical foundation on which this phenomenon can begin to be understood.

The goal of this work is not to advocate for any particular aggregation method or
implementation but to present a variety of approaches and analyze structure,
runtime, and reliability that result when they are applied to nonlinear optimization
problems.

\section{Mathematical background}

\subsection{Nonlinear optimization}
\label{sec:nlopt}
We address nonlinear optimization problems in the form given by
\rev{Problem} (\ref{eqn:nlopt}), where $z$ is a variable vector in $\mathbb{R}^n$
and $f$, $g$, and $h$ are functions with outputs in
$\mathbb{R}$, $\mathbb{R}^{m_g}$, and $\mathbb{R}^{m_h}$.
\begin{equation}
  \begin{array}{cl}
    \displaystyle\min_z & f(z) \\
    \text{s.t.} & g(z) = 0 \\
    & h(z) \leq 0 \\
    & z^l \leq z \leq z^u \\
  \end{array}
  \label{eqn:nlopt}
\end{equation}
Constant vectors $z^l$ and $z^u$ may have coordinates that are
$-\infty$ or $+\infty$, respectively, if $z$ is unbounded in some coordinates.
Lagrange multipliers for equality, inequality, lower bound, and upper bound constraints
are denoted $\lambda$, $\nu$, $\gamma^l$, and $\gamma^u$.
For instance, interior point methods are efficient algorithms for arriving at local solutions
of problems in this form \cite{biegler2010}.
They operate by iteratively solving the linear system given in Equation \ref{eqn:kkt}
to compute search directions $(dz, ds, d\lambda, d\nu)$.
\rev{
\begin{equation}
  \resizebox{\textwidth}{!}{
    $
  \left[\begin{array}{cccc}
    W + \Sigma & & \nabla g & \nabla h \\
    & S^{-1} \Gamma^s & & I \\
    \nabla g^T & & & \\
    \nabla h^T & I & & \\
  \end{array}\right]
  \left(\begin{array}{c}
    dz \\
    ds \\
    d\lambda \\
    d\nu \\
  \end{array}\right)
  = -\left(\begin{array}{c}
    \nabla f + \nabla g^T\lambda + \nabla h^T\nu 
    + \left(Z_u^{-1} - Z_l^{-1}\right)\mu \mathbbm{1} \\
    \nu - S^{-1} \mu\mathbbm{1} \\ 
    g \\
    h + s\\
  \end{array}\right)
  $
}
  \label{eqn:kkt}
\end{equation}
}
Here, $W$ is the Hessian of the Lagrangian with respect to $z$ and
$\Sigma = Z_u^{-1}\Gamma^u + Z_l^{-1}\Gamma^l$.
Diagonal matrix $S$ has a diagonal of slack variables $s$, which are
introduced along with the bound constraint $s\geq 0$ and multiplier $\gamma^s$
to reformulate inequality constraints. $Z_u$ and $Z_l$ are
diagonal matrices of $(z^u - z)$ and $(z - z^l)$, $\Gamma^s$,
$\Gamma^u$, and $\Gamma^l$ are diagonal matrices of the corresponding
Lagrange multipliers, and $\mu$ \rev{is} a barrier parameter chosen by the algorithm.
The matrix on the left-hand-side of this equation is referred to as the
Karush-Kuhn-Tucker, or KKT, matrix.
Our convention for Jacobian matrices, e.g., $\nabla g$,
is that rows correspond to variables while columns correspond to constraints.
We note that the transpose of the Jacobian, e.g., $\nabla g^T$, has the same
sparsity pattern as the incidence matrices we will describe in Section \ref{sec:bipartite}.

Constructing the linear system in Equation \ref{eqn:kkt} requires the
Hessian of the Lagrangian, the Jacobian of constraints, and the gradient
of the objective function.
If algebraic expression graphs for constraint and objective functions are available,
these derivatives are computed by automatic differentiation (AD) of these functions
\cite{griewank2008}.
For most problems, the computational bottleneck of solving \rev{Problem} (\ref{eqn:nlopt})
with interior point methods is factorization of the KKT matrix.
However, for some problems with complicated algebraic expressions, evaluation of
the Hessian of the Lagrangian, $W$, is the bottleneck.

\subsection{Algebraic variables and constraints}
Section \ref{sec:nlopt} considers vector-valued functions of a variable vector $z$.
In algebraic modeling language interfaces to nonlinear optimization solvers, however,
variables are represented as scalar quantities that participate in algebraic expressions.
These algebraic expressions define each coordinate of the functions in \rev{Problem}
(\ref{eqn:nlopt}).
An algebraic expression can be represented programmatically as a tree 
where leaf nodes are (scalar) variables or constants and all other
nodes are algebraic operators such as addition ($+$), subtraction ($-$), or multiplication ($*$).
For example, the algebraic expression tree for the expression $3x^2-1$ is given in Figure (\ref{fig:expression-tree}).

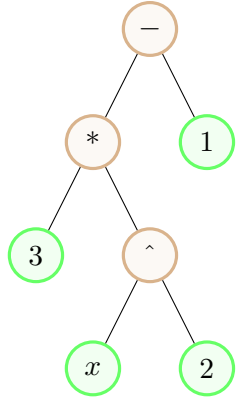
\begin{figure}
    \begin{center}
        \begin{tikzpicture}[
            leaf/.style={circle, draw=green!60, fill=green!5, very thick, minimum size=7mm},
            op/.style={circle, draw=brown!60, fill=brown!5, very thick, minimum size=7mm},
          ]
          \node[op] (name) {\large $-$}
          child {node[op] {\large *}
            child {node[leaf] {\large $3$}}
            child {node[op] {\large $\hat{}$}
              child {node[leaf] {\large $x$}}
              child {node[leaf] {\large $2$}}
              }
            }
            child {node[leaf] {\large $1$}}
          ;
        \end{tikzpicture}
    \end{center}
    \caption{Algebraic expression tree for the expression $3x^2 -1$}
    \label{fig:expression-tree}
\end{figure}
Given an algebraic expression graph defining a function, useful information about
the function can be computed efficiently by recursively visiting the nodes in the
graph (Gay provides a good overview in \cite{gay2012}).
If common subexpressions are reused among multiple functions (or even within the same
function), expressions may be represented as a directed acyclic graph (DAG), where
subgraphs corresponding to common subexpressions are not duplicated.

For the aggregation methods described in Section \ref{sec:methods}, we require
subroutines that process expression graphs and return (a) the set of variables
that participate and (b) the set of variables that participate linearly.
We denote these subroutines {\tt all\_vars} and {\tt linear\_vars}, each
of which accept a scalar constraint and return the set of scalar variables
that participate (linearly). These subroutines run in $\mathcal{O}(l)$ time,
where $l$ is the number of nodes in the expression graph.
\begin{table}
  \centering
  \caption{Subroutines used by aggregation methods}
  \resizebox{\textwidth}{!}{
    \begin{tabular}{cp{0.2\linewidth}p{0.2\linewidth}cp{0.3\linewidth}}
      \toprule
    Subroutine & Inputs & Outputs & Time complexity & Description \\
    \midrule
    {\tt size} & Set & Integer & $\mathcal{O}(1)$ & Return the number of elements in the set \\
    \midrule
    {\tt all\_vars} & Scalar constraint & Set of variables & $\mathcal{O}(l)$ & Return the set of variables that participate in a constraint\\
    {\tt linear\_vars} & Scalar constraint & Set of variables & $\mathcal{O}(l)$ & Return the set of variables that participate linearly in a constraint\\
    \midrule
    {\tt bipartite\_graph} & Set of variables, set of constraints & Bipartite graph & $\mathcal{O}(l(n_v+n_e))$ & Return the bipartite graph of variable-constraint incidence \\
    {\tt linear\_bipartite\_graph} & Set of variables, set of constraints & Bipartite graph & $\mathcal{O}(l(n_v+n_e))$ & Return the bipartite graph of linear variable-constraint incidence \\
    {\tt induced\_subgraph} & Graph, set of edges & Graph & $\mathcal{O}(n_e)$ & Return the subgraph induced by a set of edges \\
    {\tt maximum\_matching} & Bipartite graph & Set of edges & $\mathcal{O}( (n_v + n_e) \sqrt{n_v})$ & Return a maximum-cardinality matching \\
    {\tt block\_triangularize} & Bipartite graph, perfect matching & Ordered set of sets of edges & $\mathcal{O}(n_v + n_e)$ & Return the partition of the matching that defines the irreducible block-triangular form \\
    \bottomrule
  \end{tabular}
  }
  \label{tab:subroutines}
\end{table}

\subsection{Bipartite graph of variables and constraints}
\label{sec:bipartite}
While expression graphs precisely represent the functional form of constraints,
bipartite graphs of variables and constraints let us compute useful properties
of sets of constraints (and the variables they contain).

We define a {\it graph}, $G=(V,E)$, as a set of vertices (nodes), $V$, and a set of edges,
$E$. An edge $e=(u,v)$ is a pair of nodes.
A {\it bipartite graph}, $G=(A,B,E)$, contains two disjoint sets
of nodes $A$ and $B$ where, for every edge $(a,b) \in E$, $a$ is in $A$
and $b$ is in $B$.
A node $v$ and edge $e$ are {\it incident} if $e$ contains $v$.
The {\it degree} of a node is its number of incident edges.
A {\it subgraph} induced by a subset of nodes $V_s\subset V$ is the graph
$(V_s, E_s)$, where $E_s$ is the set of edges with both nodes in $V_s$.
A subgraph induced by a subset of edges $E_s\subset E$
is the subset induced by the set of nodes $V_s$, where 
$V_s$ is the set of incident nodes of edges in $E_s$.

A bipartite graph may be represented as an {\it incidence matrix}, where
rows correspond to one set of nodes and columns correspond to the other.
An entry exists for the sparse matrix row-column pair if an edge exists between
the corresponding nodes in the bipartite graph.
\begin{figure}[h]
  \centering
  \includegraphics[width=12cm]{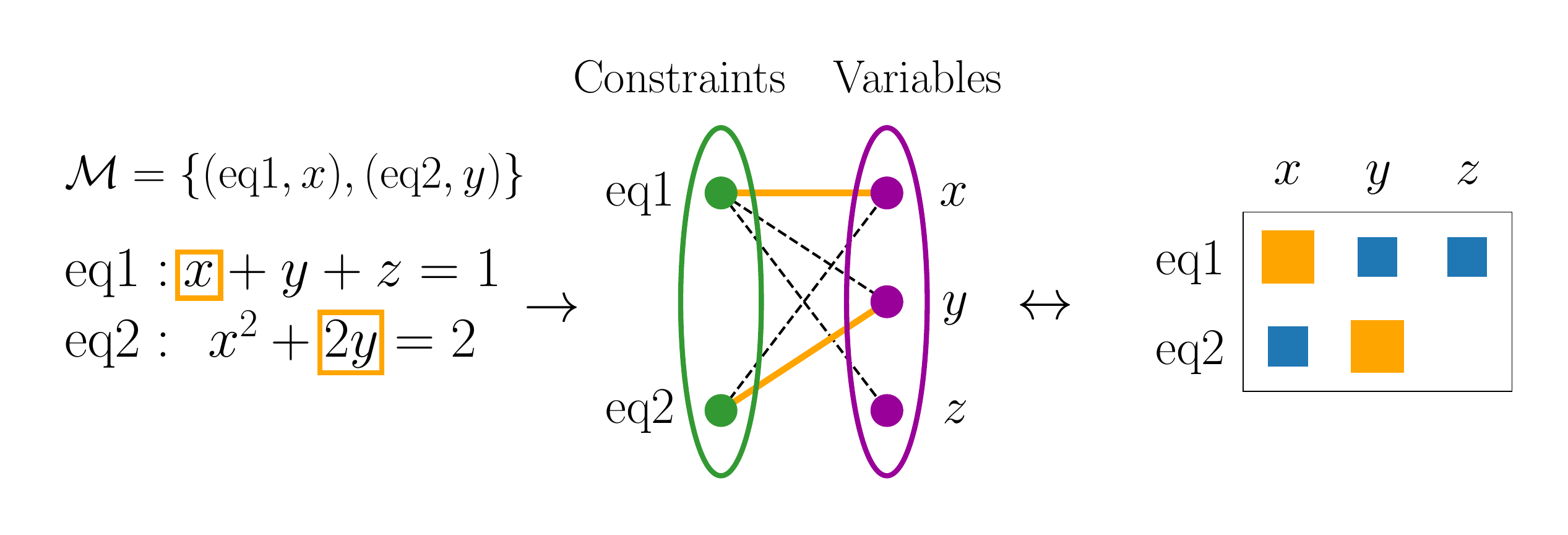}
  \caption[]{Bipartite graph and incidence matrix corresponding to variables and equations
  with a maximum matching highlighted.\footnotemark[1]}
  \label{fig:igraph-matching}
\end{figure}
\footnotetext[1]{Figure reproduced from \cite{parker2023dulmage}.}

A {\it matching}, $\mathcal{M}$, is a set of edges, no two of which share a node.
A {\it maximum matching} is a matching of the largest cardinality possible for
a given graph, and a {\it perfect matching} is a matching that covers every
node in the graph. A maximum matching in a bipartite graph may be computed in
$\mathcal{O}( (n_v + n_e) \sqrt{n_v})$ time, where $n_v$ is the number of nodes and
$n_e$ is the number of edges, by the algorithm of Hopcroft and Karp
\cite{hopcroft1973}.
Not all bipartite graphs have a perfect matching, but a subgraph induced by
a matching $\mathcal{M}$ always has a perfect matching ($\mathcal{M}$ itself).
A maximum matching is not unique, even if it is perfect; that is, many different
matchings may have the same (maximum) cardinality.
\begin{figure}[h]
  \centering
  \includegraphics[width=9cm]{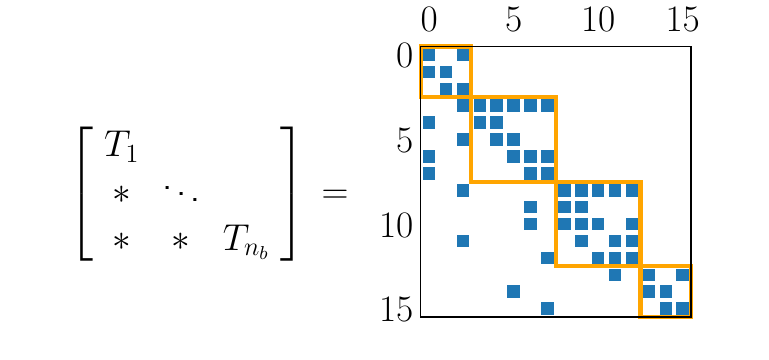}
  \caption[]{Block triangular form of an incidence matrix.
    Here, $T_1,\dots,T_{n_b}$ are the submatrices corresponding
    to subgraphs induced by subsets in the block triangular partition.\footnotemark[2]
  }
  \label{fig:bt}
\end{figure}

The incidence matrix of a bipartite graph that admits a perfect matching can be
permuted into irreducible block lower triangular form by an algorithm given by
Duff and Reid \cite{duff1978implementation,duff1978algorithm}.
This is done by a subroutine, called {\tt block\_triangularize} in Table
\ref{tab:subroutines}, that accepts a bipartite graph and a perfect matching thereof
and returns an ordered partition $\mathcal{T}$ of the perfect matching. Permuting the
rows and columns of the incidence matrix into the order defined by the partition will
put it into irreducible block lower triangular form. Each subset of edges
$T\in \mathcal{T}$ corresponds to a diagonal block of the incidence matrix.
The diagonal blocks are irreducible in the sense that submatrices defined by each
$T\in \mathcal{T}$ cannot be further partitioned into block triangular form.
That is, if $\mathcal{T}_T$ is the block triangular partition of block $T$,
${\tt size}(\mathcal{T}_T)=1$. This property of irreducibility is known as the
{\it Strong Hall Property} \cite{scott2023,coleman1986}.
An illustration a block triangular partition is given in Figure \ref{fig:bt},
and details about the block triangularization algorithm are given in Appendix
\ref{sec:block-triang}.
\footnotetext[2]{Figure adapted from \cite{parker2023dulmage}.}

Given a set of scalar variables, $X$, and a set of scalar constraints, $C$,
a bipartite graph may be constructed where one set of nodes is the set
of variables and the other is the set of constraints. An edge exists
between a variable node $x$ and a constraint node $c$ if $x$ participates
in $c$.
Figure \ref{fig:igraph-matching} illustrates a bipartite graph and incidence matrix
corresponding to three-variable, two-constraint system.
Alternatively, the edges may be restricted to only the pairs
corresponding to linear variable-constraint incidence.
An edge corresponding to linear variable-constraint incidence is referred to as a
{\it linear edge}, while any other edge is a {\it nonlinear edge}.
Subroutines that return or accept bipartite graphs of variables and constraints
are used by the aggregation methods in Section \ref{sec:methods},
and are listed in Table \ref{tab:subroutines}, where $n_v$ and $n_e$ are the numbers
of vertices and edges in a graph.

If the bipartite graph of a set of variables and constraints admits a perfect matching,
it can be used to determine whether the variables and constraints can be permuted
to have a lower triangular Jacobian. Such a permutation exists if and only if the
irreducible block triangular partition $\mathcal{T}$ contains no subsets with
cardinalities greater than one. This is due to the irreducibility of the subsets
of the block triangular partition.

\subsection{Variable aggregation}
\label{sec:aggregation}
Given an optimization problem in the form of \rev{Problem} (\ref{eqn:nlopt}), it may be
possible to partition the equality constraint functions $g=(\bar g, \tilde g)$
and variables $z=(u, v)$ such that the Jacobian $\nabla_v \tilde g$ is always
nonsingular (over a domain of interest).
In this case, by the implicit function theorem, there exists a function
$v = \tilde g_v(u)$ that satisfies $\tilde g(u, \tilde g_v(u))=0$.
The function $\tilde g_v$ is exploited in reduced space optimization methods
such as \cite{parker2022,pacaud2022,pacaud2024} to eliminate
variables and equality constraints (here, $v$ and $\tilde g(u, v) = 0$) from the
optimization problem.
The resulting reduced-space optimization problem is given by \rev{Problem} \ref{eqn:reduced-nlopt}.

\begin{equation}
  \begin{array}{cl}
    \displaystyle\min_u & f(u, \tilde g_v(u)) \\
    \text{s.t.} & \bar g(u, g_v(u)) = 0 \\
    & h(u, g_v(u)) \leq 0 \\
    & u^l \leq u \leq u^u \\
    & v^l \leq g_v(u) \leq v^u \\
  \end{array}
  \label{eqn:reduced-nlopt}
\end{equation}

A subset of variables $v$ and constraints $\tilde g(u,v)=0$ to eliminate
is represented programmatically as a matching, or a set of variable-constraint
edges in the bipartite incidence graph. We refer to a matching that is used
for variable aggregation as an {\it elimination order} or an
{\it aggregation set}. While we refer to Problem \eqref{eqn:reduced-nlopt}
as a {\it reduced-space} optimization problem, it may have more constraints
(where ``constraint'' is defined to exclude variable bounds)
than the original problem. This is because bounds on the eliminated variables,
$v$, in the original problem are converted to inequality constraints in the
reduced-space problem. While aggregating a variable eliminates an
{\it equality} constraint, it may introduce up to two additional
{\it inequality} constraints if the variable has upper and lower bounds.

%
Identifying whether a general subset of variables and constraints $v$ and $\tilde g(u,v)=0$
is nonsingular over a domain of interest involves solving a global optimization
problem, which is beyond the scope of our intended presolve application.
However, it is sufficient (but not necessary) for the constraints to have the form
given by Lemma \ref{lma:semilinear}, \rev{which uses the following definition:}
\rev{\begin{definition}
  A matrix $A\in\mathbb{R}^{n\times n}$ is \textit{strictly lower triangular} if
  $A_{ij}=0$ for all $i\leq j$.
\end{definition}}
\noindent\rev{That is, a strictly lower triangular matrix is lower triangular with an empty diagonal.}
\begin{lemma}
  Let
  \[\tilde g(u, v) = v - \tilde g^{\rm def}(u, v) = 0\]
  be a subset of constraints in \rev{Problem} \ref{eqn:reduced-nlopt},
  where $\nabla_v \tilde g^{\mathrm{def},T}$ is strictly lower triangular.
  Then $\nabla_v \tilde g^T$ is nonsingular everywhere and the subsets of variables
  and constraints $v$ and $\tilde g$ may be eliminated to form a reduced-space
  optimization problem.
  \label{lma:semilinear}
\end{lemma}
\begin{proof}
  $\nabla_v \tilde g^T$ is lower triangular with a diagonal of all ones,
  and therefore is nonsingular.
\end{proof}

A subset of constraints can easily be checked for this form by inspecting algebraic
expression graphs and the bipartite incidence graph.
Here, $v$ is the vector of {\it defined variables} and $g^{\rm def}$ is
the {\it defining function} (composed of {\it defining expressions}).

Constraints with this form have the additional advantage that a reduced-space
optimization problem may be constructed explicitly by recursively substituting
each eliminated variable $v_i$ with the corresponding coordinate of the defining
function $\tilde g^{\rm def}_i$.
Because of strict lower triangularity of $\rev{\nabla}_\rev{v} \tilde g^{{\rm def}\rev{,T}}$, the eliminated
variable $v_i$ does not appear nonlinearly in the equation used to eliminate it,
so it is never necessary to solve for $v_i$ by an iterative (or implicit) equation
solving method.

While one can easily check whether constraints have the form
of Lemma \ref{lma:semilinear},
in a general optimization problem there are a combinatorial number
of subsets of variables and constraints that could have this form.
There are several approaches one can take to decide which subset to choose
for aggregation. One approach, used by AMPL with the {\tt substout=1} option,
is to try to eliminate as many constraints as possible without breaking
strict lower triangularity of $\nabla_v \tilde g^{\mathrm{def},T}$.
A more conservative approach, taken by AIMMS and Pyomo, is to only eliminate
linear equality constraints with at most two variables.
We refer to this as a {\it linear-degree-2} elimination strategy.

Other approaches could be considered. For instance, one can set an upper bound
on the number of variables in an eliminated constraint that is greater than two.
Nonlinear constraints with only certain functional forms, such as low-order
polynomials, could be considered. Additionally, the ``target constraints''
into which $\tilde g^{\rm def}_i$ is substituted could be inspected to avoid
aggregations that convert linear constraints into nonlinear constraints
(as suggested by Amarger et al. \cite{amarger1992}).
In this work, we consider aggregation strategies involving constraints with
at most two variables as well as aggregation strategies which attempt to eliminate
the maximum number of variables. The specific strategies we implement
are covered in Section \ref{sec:methods}.

\section{Aggregation methods}
\label{sec:methods}

While aggregating the maximum number of variables may be valuable for some
problems, making significant modifications to problem structure may
compromise solve time or convergence reliability for others.
For this reason, we also propose aggregation strategies that preserve certain
properties of the Jacobian of remaining constraints $\bar g$ and $h$.
This section details the variable aggregation methods that we compare in
Section \ref{sec:results}. Each elimination method is given as an algorithm
that accepts sets of all variables $X$ and all constraints $C$ and returns
a set of variables and constraints to aggregate as a matching $\mathcal{M}$
of the corresponding nodes in the bipartite incidence graph.

\subsection{Approximate maximum elimination}
A natural approach is to aggregate as many variables and constraints
that have the form of Lemma \ref{lma:semilinear} as possible.
This is similar to the approach suggested by Christensen \cite{christensen1970}.
However, identifying the maximum aggregation set (according to the form given
by \ref{lma:semilinear}) is NP-complete.
\begin{theorem}
  Let $X$ be a set of scalar variables, $C$ a set of scalar algebraic
  constraints, and $G$ the corresponding bipartite incidence graph.
  Let $\mathcal{M}$ be a matching of linear edges of $G$ where the incidence matrix
  of the subgraph induced by $\mathcal{M}$, $G_\mathcal{M}$, is lower triangular.
  We refer to the problem of identifying $\mathcal{M}$ with maximum cardinality
  as the \textit{maximum aggregation set problem}.
  The maximum aggregation set problem is NP-complete.
  \label{thm:np-complete}
\end{theorem}
\begin{proof}
  The proof is by reduction to the minimum tearing problem.
  Let $X$ and $C$ be sets of variables and {\it linear} constraints with the same
  cardinalities that admit a perfect matching.
  It is sufficient to show that identifying a maximum-cardinality matching
  of $X$ and $C$ with a lower triangular induced submatrix is NP-complete.
  Identifying this maximum-cardinality matching is equivalent to \rev{identifying
  the largest} square submatrix of $X$ and $C$'s incidence matrix that is lower
  triangular, which is the tearing problem as described by Carpanzano
  \cite{carpanzano2000order}. Carpanzano proves that this problem is NP-complete
  by reduction to the minimum feedback arc set problem (Problem 8 on Karp's list
  of NP-complete problems \cite{karp1972reducibility}).
  Because this problem is a special case of our maximum aggregation set problem,
  the maximum aggregation set problem is NP-complete as well.
\end{proof}
Because identifying a maximum aggregation set is NP-complete, we do not attempt
to solve it exactly, but instead develop heuristics to identify large aggregation
sets.

\subsubsection{Greedy heuristic \rev{(GR)}}
The first heuristic we implement constructs an elimination order by iterating over
constraints and checking whether there exists a variable that appears (a) linearly
and (b) does not participate in any constraint already in the elimination order.
If these two conditions are met, the variable-constraint pair is added to the
elimination order. This approach is given by Algorithm \ref{alg:greedy} and
is referred to as the {\it greedy} aggregation method, or GR for short.
By construction, Algorithm \ref{alg:greedy} returns a set of variable-constraint
pairs in the form required by Lemma \ref{lma:semilinear}.
This method in inspired by the description of the aggregation phase in AMPL's
presolver given in Chapter 18.2 of \cite{ampl}.
\begin{algorithm}
  \caption{: {\tt greedy\_aggregation\_set}}
  \begin{algorithmic}[1]
    \State {\bf Inputs:} Sets of variables $X$ and constraints $C$
    \State $\mathcal{M} = \{\}$
    \State $S = \{\}$
    \For{$c\text{~\bf in } C$}
      \For{$x \text{~\bf in }\tt linear\_vars(c)$}
        \If{$x\notin S$}
          \State $\mathcal{M} \gets \mathcal{M}\cup \{(x, c)\}$
          \State $S \gets S \cup \text{\tt all\_vars}(c)$
          \State {\bf break}
        \EndIf
      \EndFor
    \EndFor
    \State {\bf Return:} $\mathcal{M}$
  \end{algorithmic}
  \label{alg:greedy}
\end{algorithm}

Let set $\mathcal{M}$ be the set of variables and the corresponding constraints that are to be eliminated and, let set $S$ be the set of all variables that appear in the constraints in $\mathcal{M}$.
Given a set of variables $X$ and constraints $C$, the greedy heuristic algorithm consists of the following steps:
\begin{enumerate}
    \item Initialize set $\mathcal{M} = \{\}$ and set $\mathcal{S}=\{\}$.
    \item For each constraint in $C$, compute a set of variables that appear linearly in the constraint using the subroutine ${\tt linear\_vars}$.
    \item For each variable $x$ in the set of linear variables in a constraint, if $x$ is not already in $\mathcal{S}$:
      \begin{enumerate}
        \item Add the variable $x$ and the corresponding constraint $c$ to $\mathcal{M}$.
          $x$ has not been used by any previous constraint, so adding it will preserve lower triangularity.
        \item Add all the variables appearing in constraint $c$ to set $\mathcal{S}$.
          This prevents any of these variables, which would need to precede $x$ in a lower triangular
          ordering, from being added later in the algorithm.
      \end{enumerate}
\end{enumerate}


\subsubsection{Matching-based algorithm \rev{(LM)}}
\begin{figure}
    \centering
    \includegraphics[width=1\linewidth]{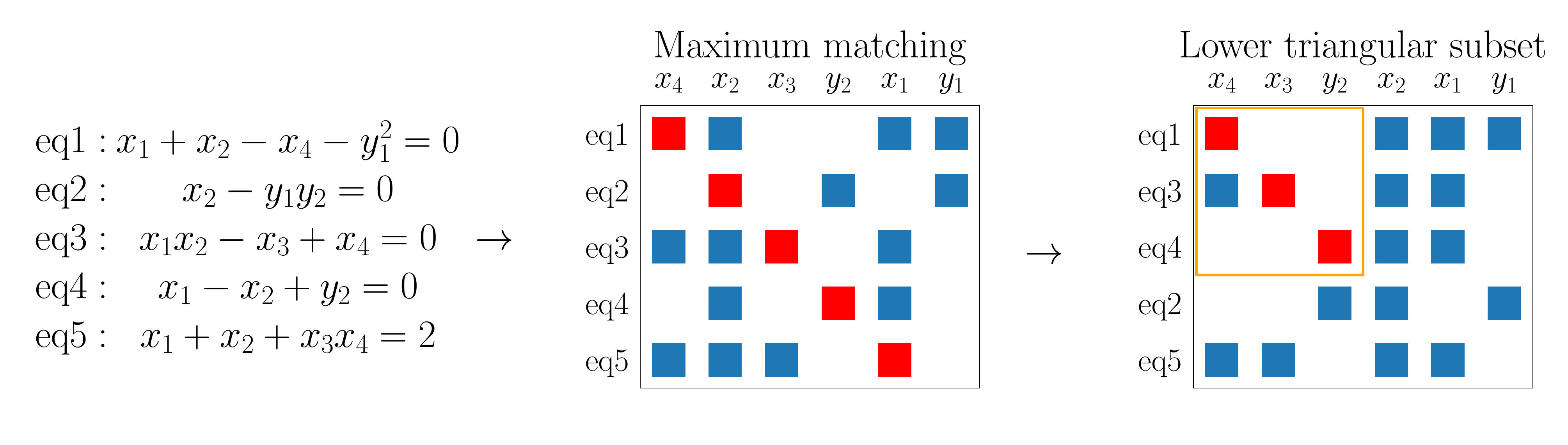}
    \caption{Illustration of the linear matching algorithm for the given equation system.
      Middle: A maximum matching of linear variable-constraint pairs
      ($\mathcal{M}_L$ in Algorithm \ref{alg:matching}).
      Right: A subset of the matching with a lower-triangular incidence matrix
      ($\mathcal{M}_T$ in Algorithm \ref{alg:matching}).}
    \label{fig:LM-strategy-schematic}
\end{figure}

The second heuristic we implement decomposes the maximum-elimination problem
into two well-studied subroutines. The first is to compute a maximum matching
on the linear bipartite graph of variables and constraints.
To generate a subset of this matching with a lower-triangular
incidence matrix, we construct the subgraph of the full bipartite graph (containing
nonlinear edges) induced by this matching, then compute the irreducible block
triangular partition of this subgraph.
In this block triangularization, any diagonal blocks of size greater
than one-by-one break lower triangularity.
To approximate the largest lower triangular subset of variable-constraint pairs 
in each of these diagonal blocks, we use the {\tt greedy}
method defined by Algorithm \ref{alg:greedy}.
We refer to this approach as the {\it linear matching} aggregation method, or LM for short.

The problem of identifying the largest lower triangular submatrix of a square, well-constrained
incidence matrix is the well-studied {\it tearing} problem
\cite{steward1965partitioning,elmqvist1994tearing}.
Sophisticated heuristics for solving this problem are available
\cite{fletcher1993ordering,tauber2014tearing}, and an exact algorithm
(that is not efficient in the worst case) has been developed for the
related minimum feedback arc set problem \cite{baharev2021mfasp}.
We use a comparatively simple heuristic and defer the integration of more
sophisticated tearing methods to future work.

Figure \ref{fig:LM-strategy-schematic} illustrates the linear-matching algorithm.
Given a set of equations, a maximum matching is computed with the matched edges highlighted in red in the incidence matrix. However, all the matched edges cannot be eliminated since eliminating both $x_1$ and $x_4$ is not possible since the expression for $x_4$ contains $x_1$ and the expression for $x_1$ contains $x_4$. Therefore, the largest lower triangular subset is computed for aggregation. The largest lower triangular subset is highlighted with the orange rectangle and the red edges indicate the variable constraint edges that will be aggregated. 

Pseudocode for the {\tt linear\_matching} method is given by Algorithm \ref{alg:matching}. Let $\mathcal{M}_T = \{\}$, be the set of variables and corresponding constraints to be eliminated. Given a set of equations with variables $\mathcal{X}$ and constraints $\mathcal{C}$, the matching based algorithm consists of the following steps:
\begin{enumerate}
    \item Compute a bipartite graph $G$ of the set of equations consisting of all edges.
    \item Compute a linear bipartite graph $G_L$ of the equations consisting of only linear edges.
    \item Compute a maximum matching $\mathcal{M}_L$ on the linear bipartite graph
    \item Compute the subgraph $G_\mathcal{M}$ of the bipartite graph $G$ induced by the maximum matching $\mathcal{M}_L$
    \item Compute a block triangular partition $\mathcal{T}$ of the subgraph induced by the matching using the $\tt block\_triangularize$ subroutine.
    \item For each block in $\mathcal{T}$:
      \begin{enumerate}
        \item If the size of the block is $1$, add the variable constraint pair in the block to $\mathcal{M}_T$
        \item If size of the block is $>1$, the $\tt greedy$ heuristic (Algorithm \ref{alg:greedy}) is used on the block to extract the variable constraint pairs to be added to $\mathcal{M}_T$.
      \end{enumerate}
\end{enumerate}
By extracting a lower triangular submatrix from each block in the block triangular partition,
the submatrix induced by the new matching $\mathcal{M}_T$ is lower triangular. Because
the block triangular partition was computed from a matching of only linear edges, $\mathcal{M}_L$,
the reduced matching $\mathcal{M}_T$ contains only linear edges as well.
The matchings $\mathcal{M}_L$ and $\mathcal{M}_T$ are illustrated in Figure \ref{fig:LM-strategy-schematic}.

The runtime of this algorithm is dominated by either the runtime of computing the
maximum linear matching $\rev{\mathcal{M}}_L$, or, if $G_L$ is small, the runtime of constructing
$G_L$ in the first place.
\begin{algorithm}
  \caption{: {\tt linear\_matching\_aggregation\_set}}
  \begin{algorithmic}[1]
    \State {\bf Inputs:} Sets of variables $X$ and constraints $C$
    \State $G = \text{\tt bipartite\_graph}(X, C)$
    \State $G_L = \text{\tt linear\_bipartite\_graph}(X, C)$
    \State $\mathcal{M}_L = \text{\tt maximum\_matching}(G_L)$
    \State $G_\mathcal{M} = \text{\tt induced\_subgraph}(G, \mathcal{M}_L)$
    \State $\mathcal{T} = \text{\tt block\_triangularize}(G_\mathcal{M}, \mathcal{M}_L)$
    \State $\mathcal{M}_T = \{\}$
    \For{$T\text{~\bf in } \mathcal{T}$}
      \If{$\text{\tt size}(T) = 1$}
        \State $\mathcal{M}_T \gets \mathcal{M}_T \cup T$
      \Else                                      
        \State $\mathcal{M}_T \gets \mathcal{M}_T \cup \text{\tt greedy}(T)$
      \EndIf
    \EndFor
    \State {\bf Return:} $\mathcal{M}_T$
  \end{algorithmic}
  \label{alg:matching}
\end{algorithm}

We note that the cardinality of the maximum matching on the linear bipartite graph,
$\mathcal{M}_L$ in Algorithm \ref{alg:matching}, is an upper bound on the number
of variables that can be aggregated using constraints in the form given by
Lemma \ref{lma:semilinear}. This algorithm therefore has the advantage
that the suboptimality of its solution, relative to a true maximum-cardinality
aggregation, can be bounded.

We also note that $\text{\tt greedy}(T)$ on line 12 of Algorithm \ref{alg:matching}
returns a set of at least
one edge, as there are at least $\text{\tt size}(T)$ edges of linear variable-constraint
incidence in the bipartite graph of variables and constraints in $T$.
This follows from 
the fact that $\mathcal{T}$ is constructed from a matching $\mathcal{M}_L$ with only
linear edges.
This gives us $\text{\tt size}(\mathcal{T})$, or the number of diagonal blocks in
the block triangular form, as a lower bound on the cardinality of the aggregation
returned by this algorithm. This result is summarized in Theorem \ref{thm:matching-bounds}.

\begin{theorem}
  Let $X$ be a set of variables and $C$ a set of constraints, with bipartite graph $G$
  and linear bipartite graph $G_L$. Let $n_{\rm match}$ be the cardinality of $\mathcal{M}_L$,
  a maximum matching of the linear bipartite graph. Let $G_\mathcal{M}$ be the subgraph
  of $G$ induced by $\mathcal{M}_L$, with a block triangular partition $\mathcal{T}$
  of cardinality $n_{\rm block}$. Let $\mathcal{A}$ be the matching returned by Algorithm
  \ref{alg:matching} that is used for aggregation, with cardinality $n_{\rm agg}$.
  Then\rev{,}
  \[n_{\rm block} \leq n_{\rm agg} \leq n_{\rm match}\rev{.}\]
  \label{thm:matching-bounds}
\end{theorem}
\begin{proof}
  Each block $T\in \mathcal{T}$ contributes at least one edge to the matching $\mathcal{A}$.
  If $\texttt{size}(T)=1$, it contributes exactly one edge. If $\texttt{size}(T) > 1$, it
  contributes at least one edge as the first constraint encountered in
  {\tt greedy} contains at least one linear variable (the variable matched with
  this constraint in $\mathcal{M}_L$).
  Therefore, $n_{\rm block} \leq n_{\rm agg}$.

  Each block $T\in \mathcal{T}$ contributes at most ${\tt size}(T)$ edges to $\mathcal{A}$.
  This is by construction, as {\tt greedy} contributes at most one edge per
  constraint provided as input, and there are ${\tt size}(T)$ constraints provided as
  input. If ${\tt size}(T)=1$, then exactly ${\tt size}(T)$ edges are contributed to
  $\mathcal{A}$. The number of edges in $\mathcal{A}$ is therefore at most 
  $\sum_{T\in \mathcal{T}} {\tt size}(T)$, which equals $n_{\rm match}$ as $\mathcal{T}$
  partitions $\mathcal{M}_L$. That is,
  \[n_{\rm agg} \leq \sum_{T\in \mathcal{T}} {\tt size}(T) = n_{\rm match}\rev{.}\]
\end{proof}

We note that the lower bound in Theorem \ref{thm:matching-bounds} is specific to the
particular linear maximum matching that was computed by the {\tt maximum\_matching}
subroutine. Computing a maximum matching that induces a highly decomposable block
triangular form has been studied by Lima et al. \cite{lima2006} using metaheuristic
methods. We leave the integration of heuristic or exact methods for choosing
a ``highly decomposable matching'' to future work.

\subsection{Fixed-variable aggregation \rev{(LD1)}}\label{sec:fixed_var_elimination}
A variable is {\it fixed} if it is the only variable in a linear equality constraint.
A simple aggregation strategy is to replace fixed variables with the value defined by
their single-variable constraints. Algorithm \ref{alg:fixed_variable}
describes the steps used to generate an ordered set of variables and constraints that
can be eliminated. The {\tt filter\_constraints\_fixed\_variable} subroutine filters
linear, degree-one constraints. That is, constraints of the form $y = a$ where $a$ is
a constant pass through the filter.

We refer to this algorithm as the {\it linear degree-1} aggregation method, or ``LD1'' for short.
The subsystem containing the filtered constraints is sent
to the {\tt linear\_matching} subroutine given in Algorithm \ref{alg:matching} to generate
the ordered set of variables and constraints that are eliminated.
The {\tt linear\_matching} subroutines prevents situations where two constraints are assigned to
eliminate a single variable. In this case of constraints each containing only a single variable,
this would imply either that one of the constraints is redundant or that the model is infeasible.

Because aggregating fixed variables and their defining constraints may reduce the number of
variables in other constraints, doing so may lead to more fixed variables in the model.
For this reason, this strategy is always implemented recursively until there are no fixed
variables left to eliminate.
Fixed-variable aggregation reduces the number of variables per constraint and the complexity
of expression graphs. It reduces the number of non-zeros in the Jacobian as it doesn't
introduce any additional variables into the problem.
\begin{algorithm}
  \caption{: {\tt degree\_1\_aggregation\_set}}
  \begin{algorithmic}[1]
    \State {\bf Inputs:} Sets of variables $X$ and constraints $C$
    \State $C' = \text{\tt filter\_constraints\_fixed\_variable}(C)$
    \State $\mathcal{M} = \text{\tt linear\_matching}(X, C')$
    \State {\bf Return:} $\mathcal{M}$
  \end{algorithmic}
  \label{alg:fixed_variable}
\end{algorithm}

\begin{table}
  \centering
  \caption{Subroutines used to filter constraints}
  \resizebox{\textwidth}{!}{
    \begin{tabular}{ccc}
      \toprule
    Subroutine & Inputs \& Outputs & Example acceptable constraint  \\
    \midrule
    {\tt filter\_constraints\_degree2}            & Set of constraints & $y = 2x^2$ \\
    {\tt filter\_constraints\_linear\_degree2}    & Set of constraints & $y = 2x + 3$ \\
    {\tt filter\_constraints\_equal\_coefficient} & Set of constraints & $y = x + 4$ \\
    {\tt filter\_constraints\_fixed\_variable}    & Set of constraints & $y = 1$ \\
    \bottomrule
  \end{tabular}
  }
  \label{tab:filters}
\end{table}

\subsection{Structure-preserving aggregation}

While aggregating as many variables and constraints as possible leads to a problem
with fewer variables (and often fewer constraints), the remaining constraints may
have more variables on average and more complicated nonlinear algebraic expressions.
To limit increases in density and complexity of the reduced-space problem, we
implement methods that attempt to preserve the structure of the remaining constraints.

These methods aggregate using constraints with at most two variables, as aggregating
a variable using such a constraint will never increase the number of variables in
other constraints. The structure-preserving methods differ by what other properties,
e.g. linearity, they require these defining constraints to have.
The algorithms to identify aggregation sets rely on constraint filters that are
summarized in Table \ref{tab:filters}, in order from most permissive to most
restrictive. We note that more permissive filters, e.g., {\tt filter\_constraints\_degree2},
still return ``less general'' constraints, such as a ``variable-fixing'' constraint $y=1$.

Aggregating variables using degree-2 constraints may decrease the number of variables
in other constraints. For example, aggregating
\[ y \leftarrow x^2 + 2 \]
into the constraint
\[ w = x + y \]
yields the new constraint
\[ w = x^2 + x + 2, \]
which has only two variables as opposed to the original three.
As in this example, such aggregations may lead to more two-variable constraints
in the reduced-space model.
The structure-preserving aggregation methods in this section are implemented
using the following procedure:
\begin{enumerate}
  \item Aggregate as many fixed variables as possible by recursively calling
    Algorithm \ref{alg:fixed_variable} and applying the resulting aggregation.
  \item Aggregate as many variables as possible using two-variable constraints
    by recursively applying a structure-preserving method.
\end{enumerate}

\subsubsection{Incidence-preserving aggregation \rev{(D2)}}\label{sec:no-fill agg}
The simplest structure-preserving strategy is to aggregate as many variables as possible
using constraints with at most two variables, at least one of which appears linearly.
We refer to this as the {\it degree-2} aggregation method, or D2 for short.
Because each defining constraint contains two variables, this method cannot increase
the number of variables in any constraint in the reduced-space model.
The method is described by Algorithm \ref{alg:nonlinear-degree2},
where the {\tt filter\_constraints\_degree2} subroutine
allows all degree-2 constraints to pass through the filter. The {\tt linear\_matching}
subroutine ensures that aggregated variables appear linearly in their defining constraints.
While this strategy does not increase the number of variables per constraint, it may
introduce nonlinearities into constraints that were previously linear.

\begin{algorithm}
  \caption{: {\tt degree\_two\_aggregation\_set}}
  \begin{algorithmic}[1]
    \State {\bf Inputs:} Sets of variables $X$ and constraints $C$
    \State $C' = \text{\tt filter\_constraints\_degree2}(C)$
    \State $\mathcal{M} = \text{\tt linear\_matching}(X, C')$
    \State {\bf Return:} $\mathcal{M}$
  \end{algorithmic}
  \label{alg:nonlinear-degree2}
\end{algorithm}

\subsubsection{Linearity-preserving elimination \rev{(LD2)}}
A more conservative strategy is to aggregate variables using only linear constraints
with two variables. We refer to this as the {\it linear degree-2} aggregation method,
or LD2 for short.
This method does not increase the number of variables per constraint, and cannot
convert a linear constraint into a nonlinear constraint.
\rev{This method is defined by Algorithm \ref{alg:linear-linking},
where \linebreak {\tt filter\_constraints\_linear\_degree2}}
allows linear constraints with two variables, i.e. constraints of the form $y = ax + b$
where $x$ and $y$ are scalar variables and $a$ and $b$ are constants, to pass through
the filter.
While the number of variables per constraint does not increase, constraints' linear
coefficients can change. For example, performing the aggregation
\[ y \leftarrow 100x + 1 \]
in the constraint
\[ w = 2y \]
yields the new constraint
\[ w = 200x + 2, \]
which has linear coefficients of $[1~-200]$ (when both variables are brought to the
left-hand-side) instead of the original coefficients of $[1~-2]$.
A similar algorithm is implemented in the Pyomo {\tt nl\_v2} writer with the
{\tt linear\_presolve=True} option.
\begin{algorithm}
  \caption{: {\tt linear\_degree\_two\_aggregation\_set}}
  \begin{algorithmic}[1]
    \State {\bf Inputs:} Sets of variables $X$ and constraints $C$
    \State $C' = \text{\tt filter\_constraints\_linear\_degree2}(C)$
    \State $\mathcal{M} = \text{\tt linear\_matching}(X, C')$
    \State {\bf Return:} $\mathcal{M}$
  \end{algorithmic}
  \label{alg:linear-linking}
\end{algorithm}

\subsubsection{Equal coefficient degree-2 aggregation \rev{(ECD2)}}
The most restrictive structure-preserving aggregation strategy attempts to preserve
entries in the Jacobian matrix of remaining constraints. To do this, we aggregate
variables using linear constraints with two variables, where the variables have
coefficients with the same magnitude. For example, a constraint
\[ y = x + a \]
may be used to eliminate $y$. If $x+a$ is substituted into constraints that
do not already contain $x$, Jacobian values in the reduced-space model will be
unchanged. That is, the derivative of the original constraint with respect to $x$
is the same as the derivative of the new constraint with respect to $y$.
If $x+a$ is substituted into constraints that already contain $x$,
the new constraint will have one fewer variable than the original constraint,
and the Jacobian entry corresponding to the derivative with respect to $x$
may change.
Like the linear degree-2 aggregation strategy, this strategy preserves linearity.

We refer to this as the {\it Equal coefficient degree-2} aggregation method, or ECD2 for short.
This method is described by Algorithm \ref{alg:trivial}.
Here, the {\tt filter\_constraints\_equal\_coefficient} subroutine allows constraints of the form
$y=x+a$ to pass through the filter. The subsystem formed by the filtered constraints
and the variables that appear in the filtered constraints is passed to the {\tt
linear\_matching} subroutine to obtain an ordered set of variables and constraints to
eliminate.
\begin{algorithm}
  \caption{: {\tt equal\_coefficient\_aggregation\_set}}
  \begin{algorithmic}[1]
    \State {\bf Inputs:} Sets of variables $X$ and constraints $C$
    \State $C' = \text{\tt filter\_constraints\_equal\_coefficient}(C)$
    \State $\mathcal{M} = \text{\tt linear\_matching}(X, C')$
    \State {\bf Return:} $\mathcal{M}$
  \end{algorithmic}
  \label{alg:trivial}
\end{algorithm}

\section{Test problems}
\label{sec:problems}

This section briefly describes the test problems on which we evaluate the methods
described in Section \ref{sec:methods}.
To study the effect of aggregation method on solve time, we solve nominal instances
of each of the four test problems after applying each aggregation method.
To study the \rev{effect} of aggregation method on convergence reliability, for each
of the three test problems that support parameters and each of the aggregation methods,
we sample two parameters at eleven uniformly spaced points each (121 total points)
and perform a parameter sweep.

\subsection{Distillation column dynamic optimization}
Distillation is a widely employed separation technique based on the difference in relative
volatility. Dynamic distillation column models play an important role in optimally
controlling the distillation column. The dynamic model consists of differential algebraic equations
(DAEs) for mass balance that scale with the number of trays, control horizon, and the
time discretization along with non-linear constraints linking the mole fractions. The
resulting problem is a nonlinear dynamic optimization problem \rev{containing bilinear terms and rational nonlinearities}.
The full formulation for this model (DIST) may be found in \cite{parker2022}. 

\subsection{Moving bed reactor optimal operation}
A moving bed (MB) reactor is a two-phase chemical reactor in which gas and solid material
streams flow in counter-current directions, reacting as they come in contact with
each other. The inlet flow rates of gas and solid streams may be adjusted to achieve
desired product compositions or temperatures. As a test problem, we consider a moving
bed reactor with methane and iron oxide inlet streams operating at steady state conditions.
The methane reduces the iron oxide, producing carbon dioxide and water vapor via the
reaction
\begin{equation*}
  \ch{CH4} + 12\ch{Fe2O3} \rightarrow 2\ch{H2O} + \ch{CO2} + 8\ch{Fe3O4}\rev{.}
\end{equation*}
This reactor is modeled with nonlinear, one-dimensional differential-algebraic equations
discretized along the length domain of the reactor, as described by Ostace et al.
\cite{ostace2018} and Okoli et al. \cite{okoli2020} and implemented in the IDAES
model library \cite{lee2021idaes}. \rev{The model consists of multilinear and rational nonlinearities along with terms with fractional powers.}
In this test problem, we penalize the deviation from a set-point of 95\% oxygen
carrier conversion and an inlet solid flow rate of 591 kg/s.

\subsection{Pipeline network dynamic optimization}
The gas pipeline network optimization problem (PIPE) is a constrained NLP consisting of a
system of DAEs modeled in IDAES \cite{lee2021idaes}. The goal is to minimize the gas
transmission cost while meeting the customer demands and satisfying the mass and momentum
balance equations for gas transport \cite{NAIK20231847}. \rev{The gas network model consists of nonlinear terms with ratios and fractional powers to compute power consumption in the compressors}
IDAES provides a modular approach
to model the gas pipeline flowsheet by leveraging existing unit models and material
properties. However, this approach also leads to additional redundant variables and
constraints that link the individual unit models and property packages to build the
overall flowsheet. 

\subsection{AC optimal power flow}
Alternating current optimal power flow (ACOPF) describes the task of selecting power
generator dispatch levels and bus voltages to minimize cost while satisfying the AC
power flow equations over a network \cite{cain2012}. \rev{The model consists of $sine$ and $cosine$ nonlinearities arising from the power flow equations.}
We use as a test problem a
4,917-bus case from the PGLib-OPF test set \cite{pglib}. The test problem data is
parsed and converted into a Pyomo model by Egret \cite{knueven2019}.

\section{Computational results}
\label{sec:results}

This section describes the results of applying the methods of Section \ref{sec:methods}
to the problems of Section \ref{sec:problems} before solving with IPOPT \cite{ipopt}
version 3.14.17 with function and derivative evaluations performed by the AMPL
Solver Library (ASL) \cite{asl}.
All test problems are implemented in Pyomo, and, where applicable, differential
equations are discretized using Pyomo.DAE \cite{pyomodae}.
The methods to compute aggregation sets, described in Section \ref{sec:methods},
are implemented using Incidence Analysis \cite{parker2023dulmage}, a Pyomo extension
for analyzing incidence graphs.
Structural results are presented in Section \ref{sec:structural},
runtime results are presented in Section \ref{sec:runtime},
and convergence reliability results are presented in Section \ref{sec:convergence}.

Structural results refer to solver-agnostic model properties that do not depend
on particular values of variables and parameters at which the model instance is evaluated.
Examples of structural results include the numbers of constraints and variables, and
the number of nonzero entries in the KKT matrix.
As structural results do not depend on the chosen parameter values, we compute
these results for nominal instances of each test problem with each
aggregation method applied.
Structural results are computed using Incidence Analysis, PyNumero \cite{pynumero},
and the Pyomo {\tt .nl} writer's {\tt AMPLRepn} data structure to investigate the
structure of incidence matrices, derivative matrices, and expression graphs.

Runtime results are those obtained by solving a particular instance of a model
with a particular solver (here, IPOPT).
\rev{These include solve time and number of iterations.}
To limit the impact of difficult instances, which may have extremely long solve times,
and to avoid questions about how to penalize instances that don't converge,
we compare detailed timing statistics only for nominal instances of each test problem
with each aggregation method applied.
Runtime results are obtained by timing evaluation callbacks in the CyIpopt interface
to IPOPT \cite{cyipopt}.

Convergence reliability results are computed over a sweep of parameter values
for the three test problems that support parameters (distillation, moving bed,
and pipeline optimization problems).
For each of these problems, two parameters in the optimization formulation
are chosen, along with ranges over which we would like to vary these parameters.
Parameter ranges are chosen to yield challenging optimization problems that may
not converge within the 3000-iteration limit.
For each parameter, we sample eleven uniformly spaced values for a total of
121 parameter combinations at which we solve the optimization problem.
We perform this parameter sweep after applying every aggregation method and compare
the results.

The code used to produce the results presented in this section may be obtained
at \url{https://github.com/Robbybp/variable-elimination}.
The results presented are collected on HPE ProLiant XL170r server nodes with
two Intel 2.1 GHz CPUs and 128 GB of RAM running the RHEL8 Linux distribution
and using Python 3.11.5.

\begin{table}
  \centering
  \caption{Summary of method names and algorithm numbers}
  \rev{
    \resizebox{\textwidth}{!}{
  \begin{tabular}{cccc}
    \toprule
    Method name & Shorthand name & Algorithm & Example acceptable constraint \\
    \midrule
    Linear degree-1 & LD1 & \ref{alg:fixed_variable} & $y = 1$ \\
    Equal-coefficient degree-2 & ECD2 & \ref{alg:trivial} & $y = x$ \\
    Linear degree-2 & LD2 & \ref{alg:linear-linking} & $y = 2x$ \\
    Degree-2 & D2 & \ref{alg:nonlinear-degree2} & $y = x^2$ \\
    Greedy & GR & \ref{alg:greedy} & $y = f(w,x,\dots)$ \\
    Linear-matching & LM & \ref{alg:matching} & $y = f(w, x,\dots)$ \\
    \bottomrule
  \end{tabular}
}}
  \label{tab:method-summary}
\end{table}

\subsection{Structural results}
\label{sec:structural}

We first compare the effect of variable aggregation on model structure. We define
model structure as solver-agnostic properties of the optimization model that do not
depend on particular variable values, such as numbers of variables and constraints.
A summary of model structures after aggregating with each method is given in
Table \ref{tab:structure}.
\rev{For reference, the shorthand name and algorithm number
of each method can be found in Table \ref{tab:method-summary}.}

Here, ``Var.'' is the number of variables in the model, ``Con.'' is the number
of constraints, ``Elim.'' is the number of variables that were eliminated,
``NNZ/Con.'' is the average number of (structural) nonzero entries in the Jacobian matrix per constraint,
``Lin. NNZ/Con.'' is the average number of nonzero entries in the Jacobian matrix corresponding to linear variable-constraint incidence,
and ``Hess. NNZ'' is the number of nonzeros in the Hessian of the Lagrangian.

\begin{table}[]
    \centering
    \caption{Structural properties of models after applying each aggregation method}
    \vspace{0.2cm}
    \begin{tabular}{ccccccccccc}
      \toprule
Model     & Method    & Var.   & Con.   & Elim. & NNZ/Con. & Lin. NZ/Con. & Hess. NNZ\\
\midrule
\multirow{7}{*}{DIST} & --        &  30368 &  30068 &     0 & 3.88 & 1.69 &  48032\\
                      & LD1       &  30300 &  30000 &    68 & 3.89 & 1.69 &  48000\\
                      & ECD2      &  29400 &  29100 &   968 & 3.62 & 1.68 &  47400\\
                      & LD2       &  29068 &  28768 &  1300 & 3.64 & 1.67 &  47400\\
                      & D2        &  19468 &  19168 & 10900 & 3.97 & 2.00 &  47100\\
                      & GR        &   9932 &   9632 & 20436 & 4.92 & 1.00 &  47700\\
                      & LM        &   9900 &   9600 & 20468 & 4.93 & 1.00 &  47700\\
\midrule
\multirow{7}{*}{MB}   & --        &    870 &    869 &     0 & 3.09 & 1.39 &   1869\\
                      & LD1       &    780 &    779 &    90 & 2.97 & 1.34 &   1451\\
                      & ECD2      &    777 &    776 &    93 & 2.95 & 1.34 &   1366\\
                      & LD2       &    592 &    591 &   278 & 3.24 & 1.28 &   1285\\
                      & D2        &    491 &    490 &   379 & 3.38 & 1.34 &   1145\\
                      & GR        &    300 &    299 &   570 & 5.24 & 0.80 &   1609\\
                      & LM        &    167 &    166 &   703 & 9.23 & 0.02 &   1524\\
\midrule
\multirow{7}{*}{OPF}  & --        &  61349 &  87120 &     0 & 3.06 & 2.27 &  75323\\
                      & LD1       &  58969 &  89500 &  2380 & 2.90 & 2.16 &  72981\\
                      & ECD2      &  55891 &  92578 &  5458 & 2.81 & 2.09 &  70249\\
                      & LD2       &  55567 &  92902 &  5782 & 2.80 & 2.08 &  70113\\
                      & D2        &  50650 &  97819 & 10699 & 2.65 & 1.64 &  70063\\
                      & GR        &  27152 & 121317 & 34197 & 3.20 & 2.41 & 149351\\
                      & LM        &  10396 & 138073 & 50953 & 3.93 & 0.24 &  80021\\
\midrule
\multirow{7}{*}{PIPE} & --        &  12293 &  12221 &     0 & 3.00 & 1.70 &  33506\\
                      & LD1       &  10372 &  10411 &  1921 & 2.56 & 1.72 &  11155\\
                      & ECD2      &   7281 &   7731 &  5012 & 2.67 & 1.68 &   9414\\
                      & LD2       &   5252 &   6281 &  7041 & 2.74 & 1.73 &   9414\\
                      & D2        &   4660 &   5689 &  7633 & 2.81 & 1.80 &   7708\\
                      & GR        &   5492 &   6834 &  6801 & 3.33 & 0.97 &  16627\\
                      & LM        &   1124 &   2898 & 11169 & 8.26 & 0.92 &  25393\\     
\bottomrule
    \end{tabular}
    \label{tab:structure}
\end{table}
\begin{table}[]
    \centering
    \caption{Number of aggregations compared with bounds for the linear-matching
    strategy}
    \vspace{0.2cm}
    \begin{tabular}{ccccccccccc}
\toprule
Model     & Method    & Lower bound & Number of Eliminations & Upper bound\\
\midrule
DIST      & LM        &  1868 & 20468 & 30068\\
MB        & LM        &   703 &   703 &   703\\
OPF       & LM        & 49794 & 50953 & 51488\\
PIPE      & LM        & 11169 & 11169 & 11197\\
\bottomrule
    \end{tabular}
    \label{tab:matching-structure}
\end{table}
\begin{figure}[h]
  \centering
  \includegraphics[width=0.8\textwidth]{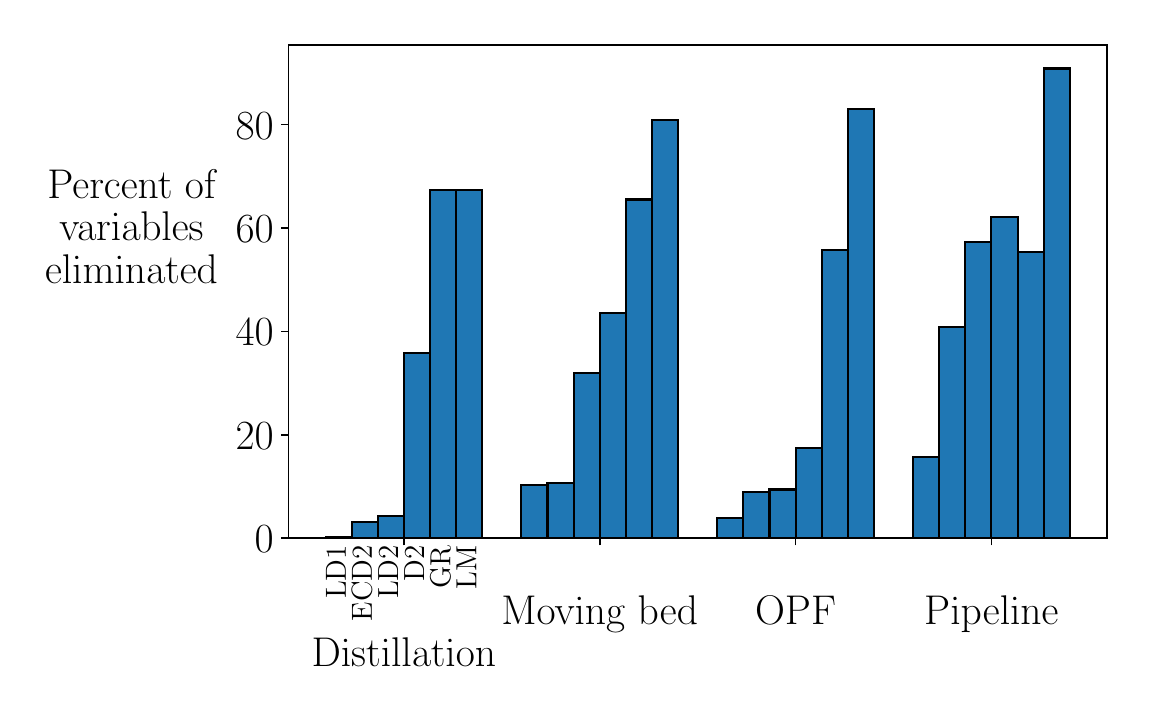}
  \caption{Percent of variables eliminated
  by each method for each model.
  Methods are in the same order as presented in Table \ref{tab:structure}.}
  \label{fig:frac-elim}
\end{figure}
We first note that the different aggregation methods eliminate significantly different
numbers of variables for each model. Simply eliminating fixed variables can remove up
to 10\% of the variables in a model (e.g. moving bed and pipeline problems), while using an
approximate-maximum aggregation strategy eliminates up to 90\% of a model's variables.
The linear-degree-2 strategy eliminates anywhere from 4\% of a model's variables
(in the distillation model) to 57\% of a model's variables (in the pipeline model).
The percent of variables eliminated by each model-method combination is illustrated in
Figure \ref{fig:frac-elim}.
For a given strategy, the wide variation in fraction of variables eliminated among
different models suggests that we may expect significantly different runtime and reliability
results when applying the same aggregation strategy to different models.
\begin{figure}[h]
  \centering
  \includegraphics[width=0.5\textwidth]{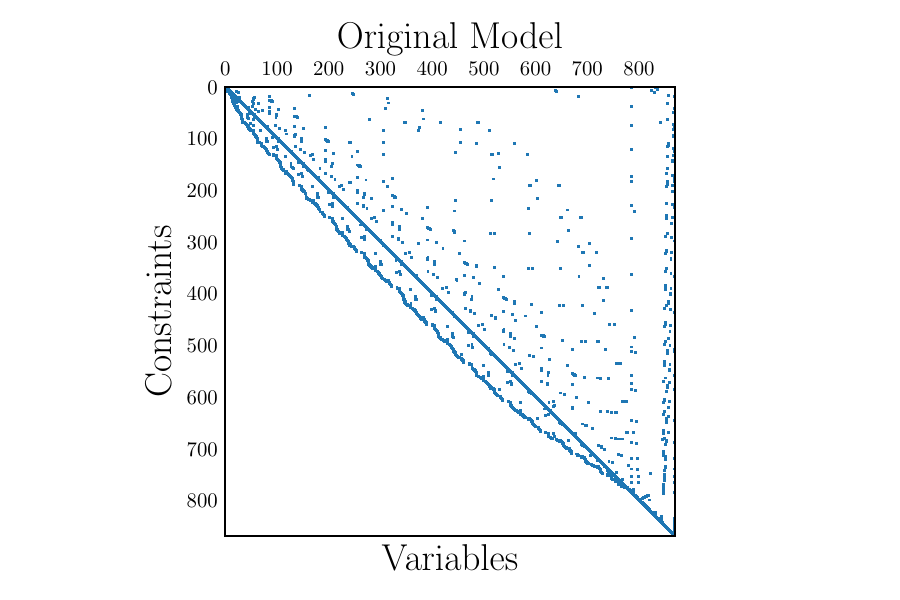}
  \hspace{-0.1\textwidth}\includegraphics[width=0.5\textwidth]{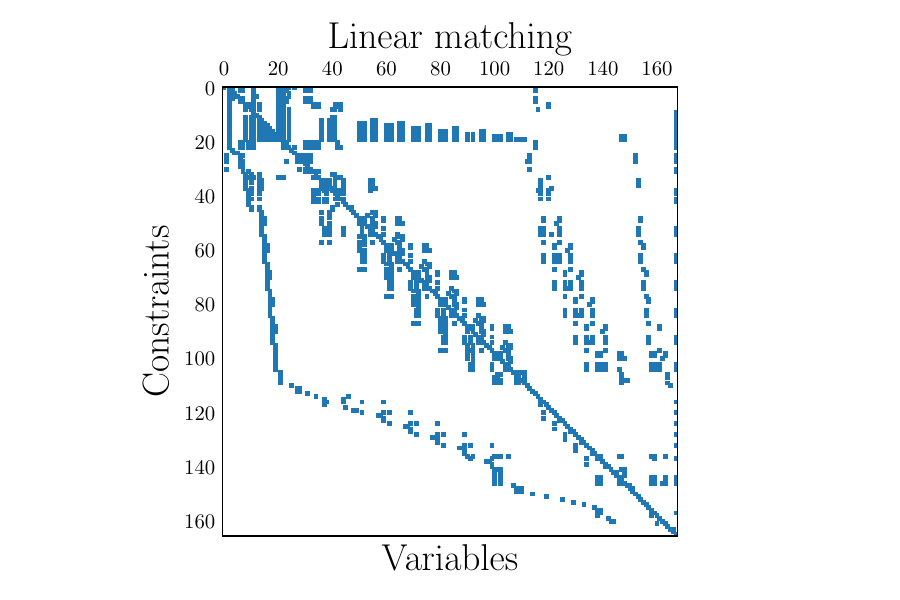}
  \caption{Incidence matrix for the moving bed optimization problem before variable aggregation and after variable aggregation using linear matching}
  \label{fig:incidence-matrix-mb}
\end{figure}

The upper and lower bounds on number of aggregations performed by the linear-matching
strategy according to Theorem \ref{thm:matching-bounds} are shown in Table \ref{tab:matching-structure}.
In most cases, the linear-maximum aggregation strategy comes close to the
upper bound on the number of eliminated variables, which is computed by performing a
maximum matching on the linear-incidence graph. Only for the distillation column
model is there a gap larger than 1\%.
For the distillation model, the block triangular form induced by the linear maximum
matching is not highly decomposable, as indicated by the relatively low lower bound.
However, the greedy procedure for recovering a lower triangular incidence matrix
recovers an aggregation set with cardinality 68\% of our upper bound.
In the moving bed model, by contrast, the linear maximum matching induces a perfectly
decomposable block triangular form. The OPF and pipeline models have maximum matchings
that lead to highly decomposable block triangular forms.
For the pipeline problem, the greedy procedure recovers only one variable and constraint
per block in the block triangularization, yielding a number of eliminations equal to
our lower bound. This is because, for the particular maximum matching computed,
all diagonal blocks in the block triangularization of size either $1\times 1$ or $2\times 2$.
In the latter case, only one variable can possibly be recovered, so the greedy
procedure for recovering a lower triangular matrix cannot improve upon the lower bound.


Increased density of Jacobian and Hessian matrices in the aggregated problem is
a concern that has been expressed by Gill et al. \cite{gill1981} and
Achterberg et al. \cite{achterberg2020}. For this reason, most of our aggregation
strategies are chosen with the intention of not increasing density in the remaining
constraints. Indeed, density stays approximately constant after applying linear-degree-1,
equal-coefficient-degree-2, linear-degree-2, or degree-2 strategies.
However, the approximate-maximum aggregation strategies do not consider any limits
on the degree of eliminated variables and constraints nodes, so they may lead to large
increases in problem density. 
In the moving bed and pipeline problems, this strategy
increases the number of variables incident on each constraint by factors of almost three,
while for the other problems, the increase in density is more modest. 
Figure \ref{fig:incidence-matrix-mb} shows the increase in density of the incidence matrix due
to variable aggregation using linear matching. Additionally, the size of the incidence
matrix decreases due to variable aggregation as variable-constraint pairs are eliminated
from the problem. Although the density of the incidence matrix
increases, it is still a relatively sparse matrix.

\begin{figure}[h]
  \centering
  \includegraphics[width=0.8\textwidth]{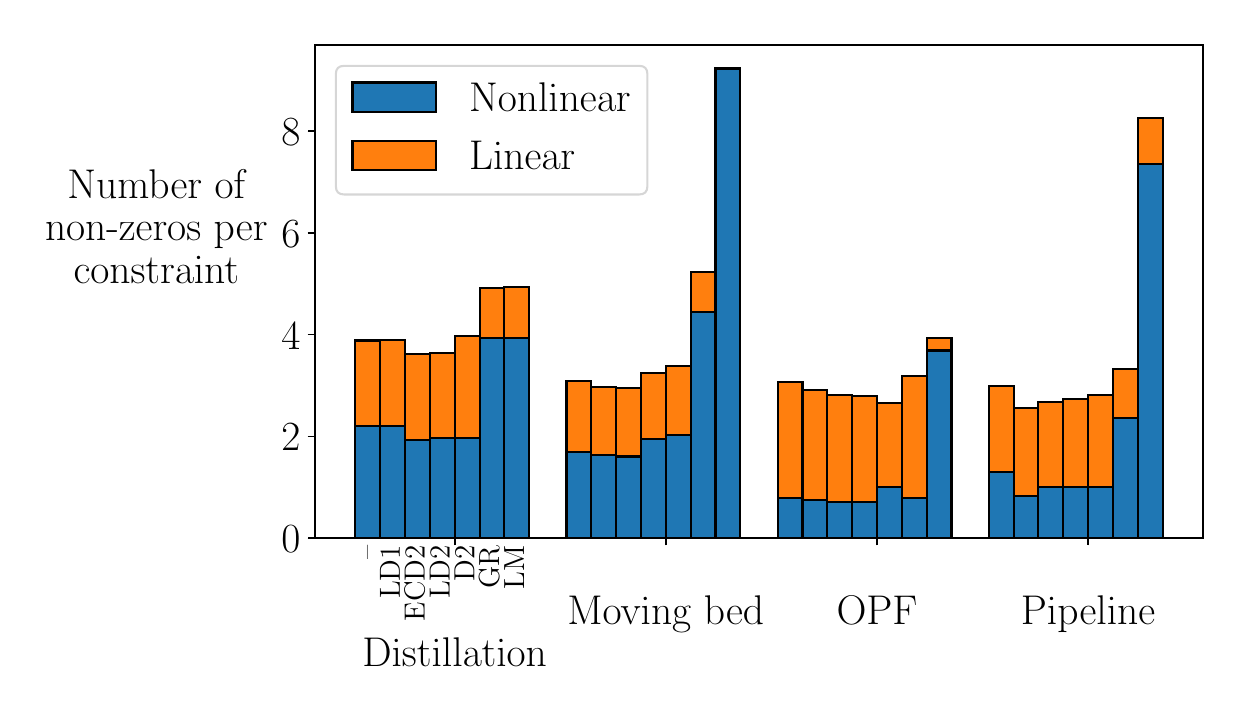}
  \caption{Number of linear and non-linear non-zeros per constraint after variable aggregation
  by each method for each model.
  Methods are in the same order as presented in Table \ref{tab:structure}}
  \label{fig:nnz-per-con}
\end{figure}

In addition to increased density of remaining constraints, aggressive variable
aggregation may result in optimization problems that are ``more nonlinear.''
We quantify the increase in nonlinearity by measuring the average number of
variables appearing linearly and nonlinearly in each constraint. Figure \ref{fig:nnz-per-con} demonstrates that for all the optimization problems, the number of linear variables per constraint
decreases significantly when the linear-matching aggregation algorithm is
used, even while the total number of nonzeros per constraint increases. On the other hand, the structure-preserving strategies approximately preserve the total non-zeros per constraint for all test problems.

\subsection{Runtime results}
\label{sec:runtime}

\begin{figure}[h]
  \centering
  \includegraphics[width=0.9\textwidth]{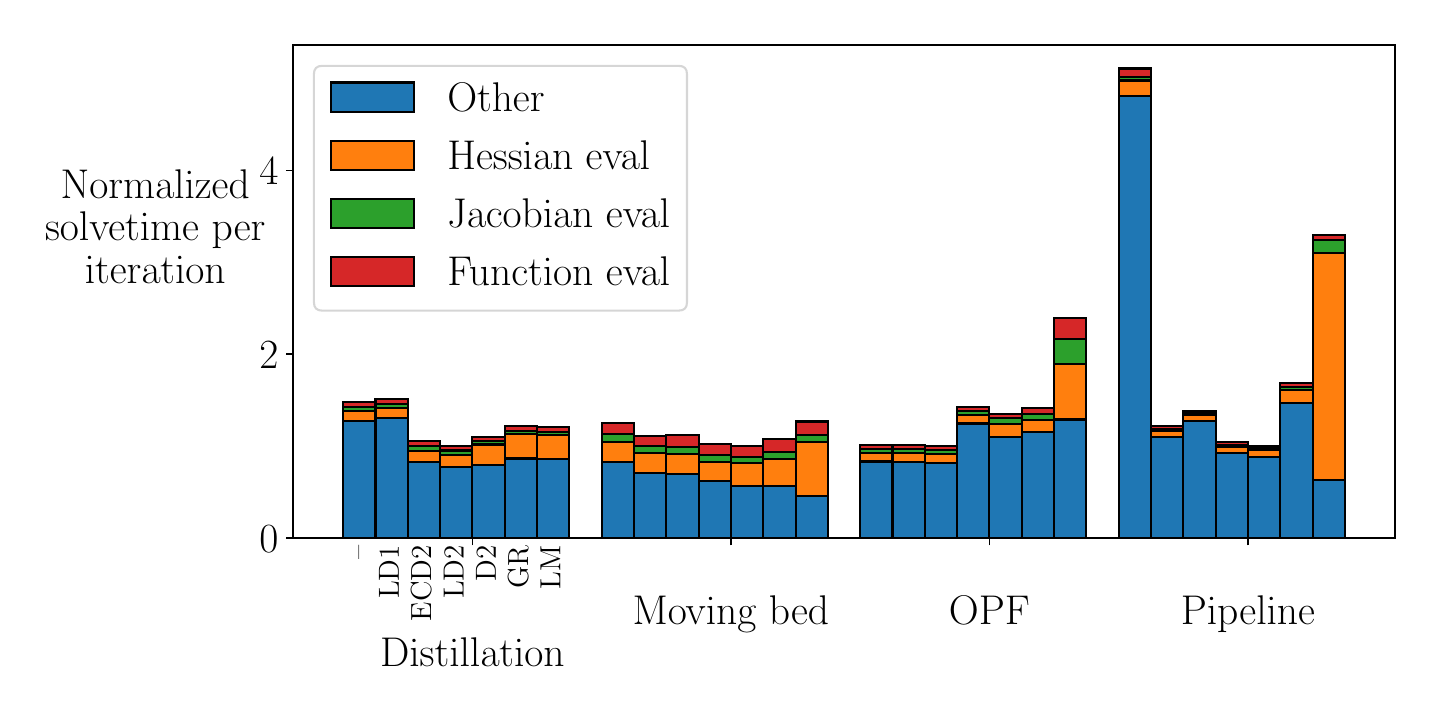}
  \caption{Solve time per iteration spent in each category by each method for each model. Methods are in the same order as presented in Table \ref{tab:solvetime}. The time is normalized by the fastest solve time so that the fastest aggregation strategy has a normalized solve time per iteration of $1.0$.}
  \label{fig:solve-time-breakdown}
\end{figure}
Variable aggregation can also have a significant effect on an optimization problem's
solve time.
We note that the direction of impact of variable aggregation on solve time is not
immediately obvious. Aggregating variables leads to a smaller, denser optimization
problem with potentially more inequality constraints (if many bounded variables are
eliminated) and potentially more nonlinear constraints (if many variables are replaced
with nonlinear expressions).
The dominant computational costs in an interior point method are function and derivative
evaluations and KKT matrix factorization. These are less expensive for a smaller problem,
but KKT matrix factorization is more expensive for a denser problem, and Hessian
evaluation is more expensive for a denser problem with more nonlinear constraints.
For these reasons, it is important to study the breakdown of runtime to understand
which factors are contributing to speedup or slow-down caused by variable aggregation.
Finally, an interior point method may take \rev{a significantly different number of iterations}
with full-space and aggregated optimization problems. While it is difficult to predict
which formulation will converge in fewer iterations, it is important to observe trends
that occur for a specific problem or problem class.

In this section we consider each problem's solve time with IPOPT \cite{ipopt} when
using each of the proposed aggregation strategies. The times spent aggregating variables,
solve times, iteration counts, and solve time breakdowns are given in Table
\ref{tab:solvetime}.

In Table \ref{tab:solvetime}, $t_{\rm build}$ is the time to build the Pyomo optimization
model, $t_{\rm elim}$ is the time to compute the sets of variables and constraints to
aggregate and perform the aggregation in-place, $t_{\rm init}$ is the time to initialize
data structures required for the optimization solve, $t_{\rm solve}$ is time
spent by IPOPT solving the optimization problem, and ``Iter.'' is the number of iterations
reported by IPOPT. ``Func.'', ``Jac.'', and ``Hess.'' are percentages of solve time spent
in callbacks used by IPOPT for function, Jacobian, and Hessian evaluations. ``Other''
is the percentage of solve time spent in IPOPT itself (rather than one of these callbacks),
which we assume is dominated by KKT matrix
factorization.\footnote[3]{This is difficult to obtain programmatically, but by inspecting IPOPT's logs we
see, in the unmodified distillation problem for example, that IPOPT spends 70\% of its time
(i.e., 70\% of the ``Other'' category)
in a routine called ``\texttt{LinearSystemFactorization}.''}
These percentages may not add up to 100 due to rounding.

\begin{table}
  \centering
  \caption{Runtime statistics of models after applying each aggregation method}
  \vspace{0.2cm}
\resizebox{\textwidth}{!}{
  \begin{tabular}{cccccccccccccc}
    \toprule
    \multirow{2}{*}{Model} & \multirow{2}{*}{Method} &\multicolumn{4}{c}{Total times (s)} & \multirow{2}{*}{Iter.} & \multicolumn{4}{c}{Solve time breakdown (\%)} & \multirow{2}{*}{Objective}\\
    & & $t_{\rm build}$ & $t_{\rm elim}$ & $t_{\rm init}$ & $t_{\rm solve}$ & & Func. & Jac. & Hess. & Other\\
\midrule
\multirow{7}{*}{DIST} & --        & \multirow{7}{*}{0.8} &    -- &   2.7 &   3.4 &  19 &   3 &   3 &   7 &  87 & 35.4\\
                      & LD1       &                      &  31.8 &   2.3 &   3.0 &  19 &   3 &   3 &   8 &  86 & 35.4\\
                      & ECD2      &                      &  38.3 &   2.4 &   1.6 &  17 &   5 &   5 &  12 &  78 & 35.4\\
                      & LD2       &                      &  38.0 &   2.4 &   1.9 &  18 &   5 &   5 &  14 &  76 & 35.4\\
                      & D2        &                      &  44.4 &   2.1 &   1.5 &  13 &   4 &   3 &  21 &  72 & 35.4\\
                      & GR        &                      &  17.7 &   2.1 &   1.5 &  13 &   5 &   3 &  24 &  69 & 35.4\\
                      & LM        &                      &  29.3 &   2.1 &   1.5 &  13 &   5 &   3 &  25 &  68 & 35.4\\
\midrule                                                                                                                      
\multirow{7}{*}{MB}   & --        & \multirow{7}{*}{4.1} &    -- &   0.5 &   0.1 &   9 &  10 &   6 &  19 &  65 & 0.00\\
                      & LD1       &                      &   2.7 &   0.5 &   0.1 &   9 &  10 &   7 &  19 &  63 & 0.00\\
                      & ECD2      &                      &   2.7 &   0.5 &   0.1 &   9 &  12 &   6 &  20 &  62 & 0.00\\
                      & LD2       &                      &   3.0 &   0.5 &   0.1 &   9 &  11 &   6 &  21 &  62 & 0.00\\
                      & D2        &                      &   3.3 &   0.4 &   0.1 &  10 &  12 &   6 &  27 &  56 & 0.00\\
                      & GR        &                      &   1.0 &   0.5 &   0.1 &  10 &  12 &   6 &  31 &  51 & 0.00\\
                      & LM        &                      &   1.5 &   0.5 &   0.1 &  10 &  11 &   5 &  49 &  35 & 0.00\\
\midrule                                                                                                                      
\multirow{7}{*}{OPF}  & --        & \multirow{7}{*}{2.8} &    -- &   4.3 &   9.8 &  61 &   4 &   4 &   9 &  82 & 1.39E+06\\
                      & LD1       &                      &  65.4 &   4.1 &   9.8 &  61 &   4 &   4 &   9 &  82 & 1.39E+06\\
                      & ECD2      &                      &  69.0 &   4.5 &  12.3 &  79 &   4 &   4 &  10 &  82 & 1.39E+06\\
                      & LD2       &                      &  76.7 &   4.5 &  44.9 & 198 &   3 &   3 &   7 &  87 & 1.39E+06\\
                      & D2        &                      &  84.5 &   4.6 &  42.5 & 200 &   3 &   4 &  10 &  82 & 1.39E+06\\
                      & GR        &                      &  30.7 &   5.0 &  16.8 &  75 &   4 &   5 &  10 &  82 & 1.39E+06\\
                      & LM        &                      &  62.5 &   8.1 &  27.4 &  73 &   9 &  11 &  26 &  54 & 1.39E+06\\
\midrule                                                                                                                      
\multirow{7}{*}{PIPE} & --        & \multirow{8}{*}{6.0} &    -- &   1.4 &   6.8 &  49 &   1 &   1 &   4 &  95 & 6.27E+03\\
                      & LD1       &                      &  21.5 &   1.1 &   1.3 &  42 &   2 &   2 &   5 &  90 & 6.27E+03\\
                      & ECD2      &                      &  24.8 &   1.2 &   1.6 &  43 &   2 &   2 &   5 &  92 & 6.27E+03\\
                      & LD2       &                      &  25.2 &   1.4 &   1.1 &  43 &   2 &   2 &   7 &  89 & 6.27E+03\\
                      & D2        &                      &  26.0 &   1.1 &   1.2 &  43 &   3 &   2 &   8 &  88 & 6.27E+03\\
                      & GR        &                      &   7.0 &   1.3 &   1.4 &  30 &   3 &   2 &   9 &  87 & 6.27E+03\\
                      & LM        &                      &  29.9 &   1.8 &   3.5 &  40 &   2 &   4 &  75 &  19 & 6.27E+03\\
\bottomrule
  \end{tabular}
}
  \label{tab:solvetime}
\end{table}

Figure \ref{fig:solve-time-breakdown} shows the normalized solve time per iteration in each category for every method and every optimization model. In two of the four problems considered (distillation and pipeline), variable
aggregation significantly reduces solve time. However, even for these two problems,
solve time does not uniformly decrease as more variables are eliminated.
For example, the pipeline model takes 3.6 s to solve with the linear-matching
strategy, but only 1.3 s to solve with the linear-degree-2, degree-2, or
greedy strategies. For the OPF model, variable aggregation generally increases
solve time, either due to larger iteration counts or an increase in the number of
inequality constraints. For the moving bed model, variable aggregation does not have
a measurable impact on solve time. 

While variable aggregation can cause a significant reduction in solve time, it can
also lead to expensive Hessian evaluations, introducing a new computational bottleneck.
For all four models, the linear-matching strategy leads to a significantly larger
fraction of solve time spent in Hessian evaluation, with a general trend that
Hessian evaluation time (as a fraction of solve time) increases as more variables
are eliminated (Figure \ref{fig:solve-time-breakdown}). When solving the pipeline model after applying the linear-matching
aggregation strategy, Hessian evaluation is responsible for 75\% of total solve time
and is the clear computational bottleneck in this solve.
We note that our implementation takes advantage of the ``defined variable'' data
structure in the ASL to not duplicate expressions when substituting an expression
for a variable many times. Increases in Hessian evaluation times are due to the
changing structure of expression graphs, rather than an increase in total size of
expression graphs.

In contrast to Hessian evaluation time, the time spent in IPOPT (``Other'' in Table
\ref{tab:solvetime}) decreases as more variables are eliminated. This suggests that
KKT matrix factorization benefits from the (generally) smaller systems in the
aggregated problems and is not significantly slowed down by the accompanying increase
in density.


\subsection{Convergence reliability}
\label{sec:convergence}
As optimization problems before and after aggregation have different constraints,
objectives, and derivative matrices, optimization algorithms may take significantly
different paths through variable-space to arrive at solutions of the two problems.
In some cases, one formulation may fail to converge within a specified iteration
limit. We say that the formulation that converges more often is more reliable.
As suggested by Parker et al. \cite{parker2022}, improving convergence reliability may
be a reason to aggregate variables in nonlinear optimization problems.

In this section, we compare convergence reliability of each of our proposed
elimination methods for distillation, moving bed, and pipeline test problems.
We do not perform a parameter sweep for the OPF test problem as Egret \cite{knueven2019}
does not construct these problems with easily mutable parameters.
Here, reliability is measured for each test problem by performing a parameter sweep
over two input parameters and counting the instances that are able to solve within
3,000 iterations with the IPOPT nonlinear solver.

Convergence results of a parameter sweep over relative volatility and the feed mole
fraction of heavy component for the distillation optimal control problem are shown,
for each method, in Figure \ref{fig:distill-convergence}.
Results of a parameter sweep varying inlet solid temperature and inlet solid flow rate
in the moving bed model are shown in Figure \ref{fig:mb-steady-convergence}
and results of a sweep varying gas temperature and supply pressure in the pipeline
dynamic optimization problem are shown in Figure \ref{fig:pipeline_convergence}.
A summary showing the percent of problems solved by each method, for each model
and in total, is shown in Table \ref{tab:convergence}. \rev{The average solve time ($t^{avg}_{solve}$) in Table \ref{tab:convergence} is computed over the successfully converged instances of the parameter sweep.}

\begin{figure}[h]
  \centering
  \resizebox{\textwidth}{!}{
  \includegraphics[width=0.38\textwidth]{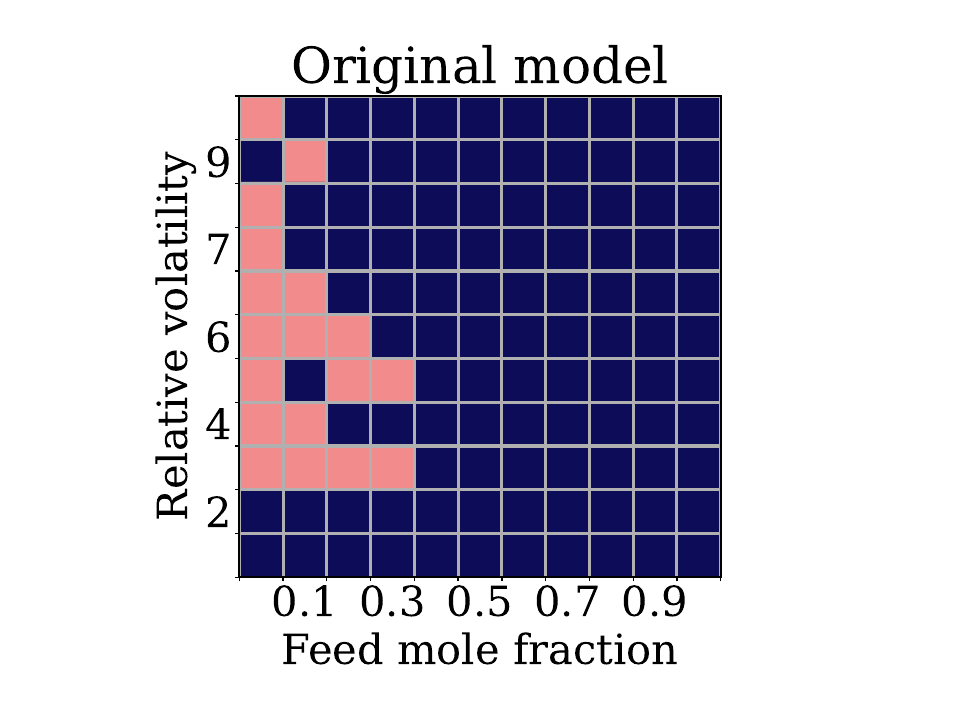}
  \hspace{-0.12\textwidth}\includegraphics[width=0.456\textwidth]{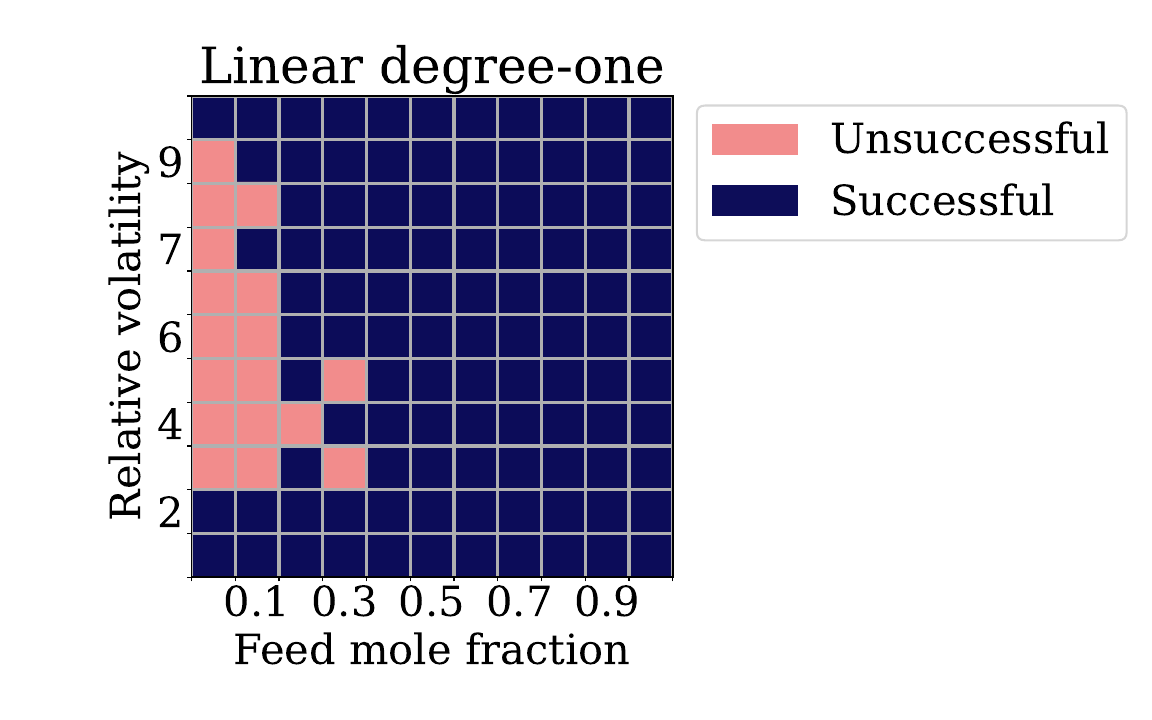}
  \hspace{-0.07\textwidth}\includegraphics[width=0.38\textwidth]{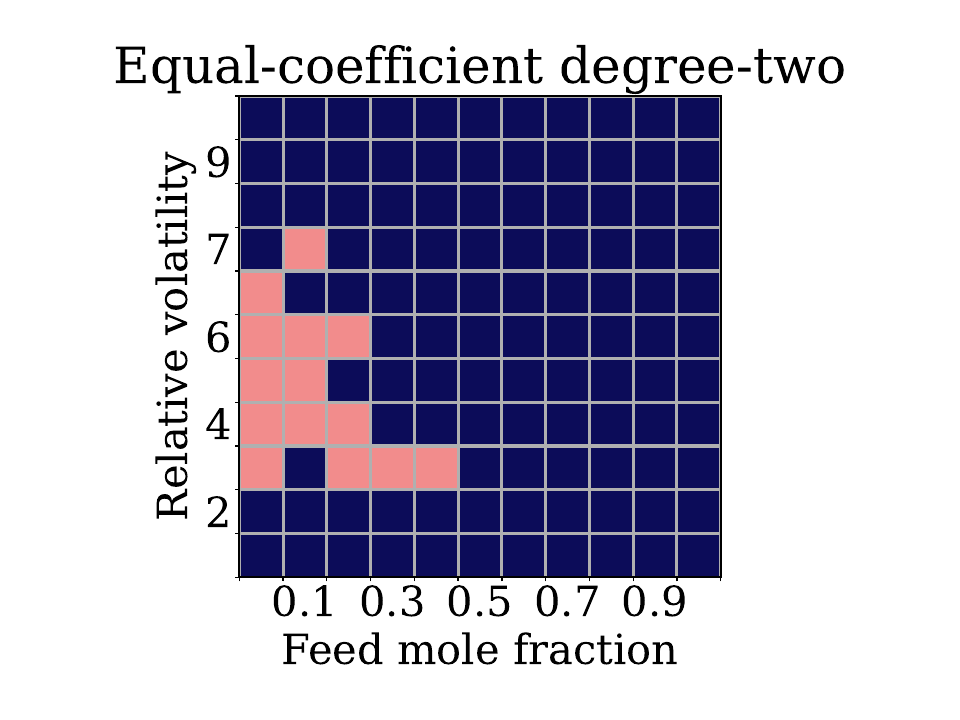}
                        }
  \resizebox{1.05\textwidth}{!}{
    \hspace{-0.025\textwidth}\includegraphics[width=0.38\textwidth]{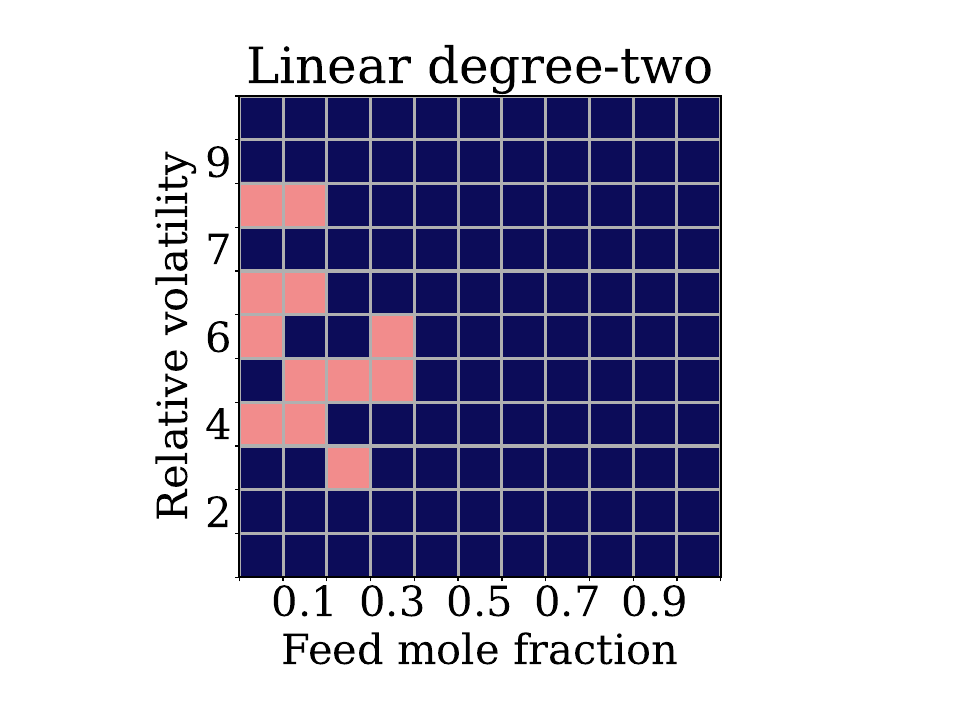}
  \hspace{-0.15\textwidth}\includegraphics[width=0.38\textwidth]{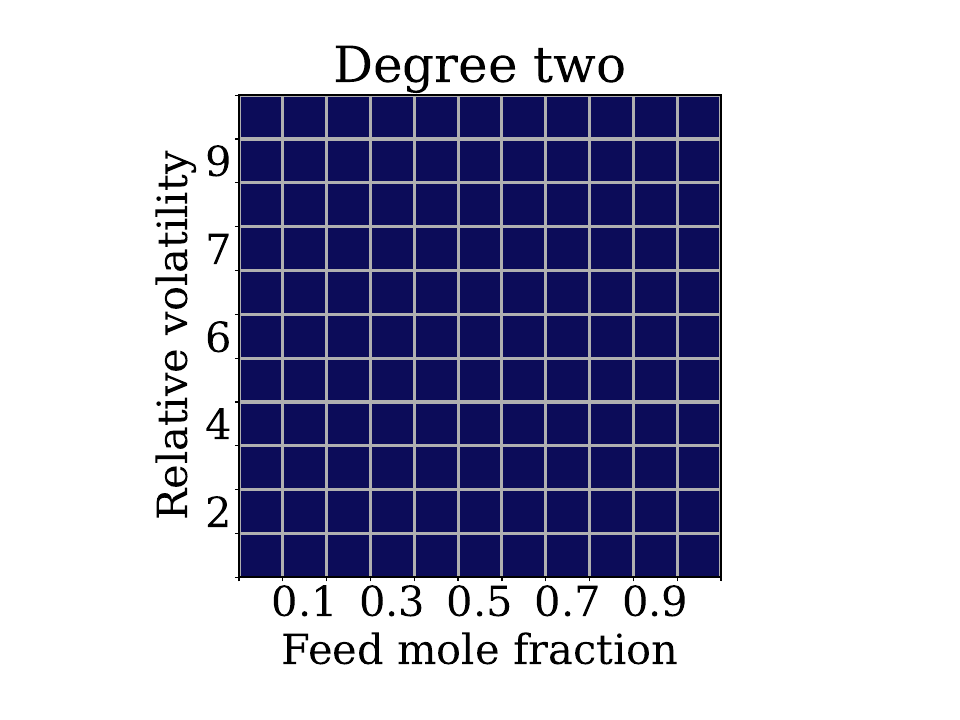}
  \hspace{-0.15\textwidth}\includegraphics[width=0.38\textwidth]{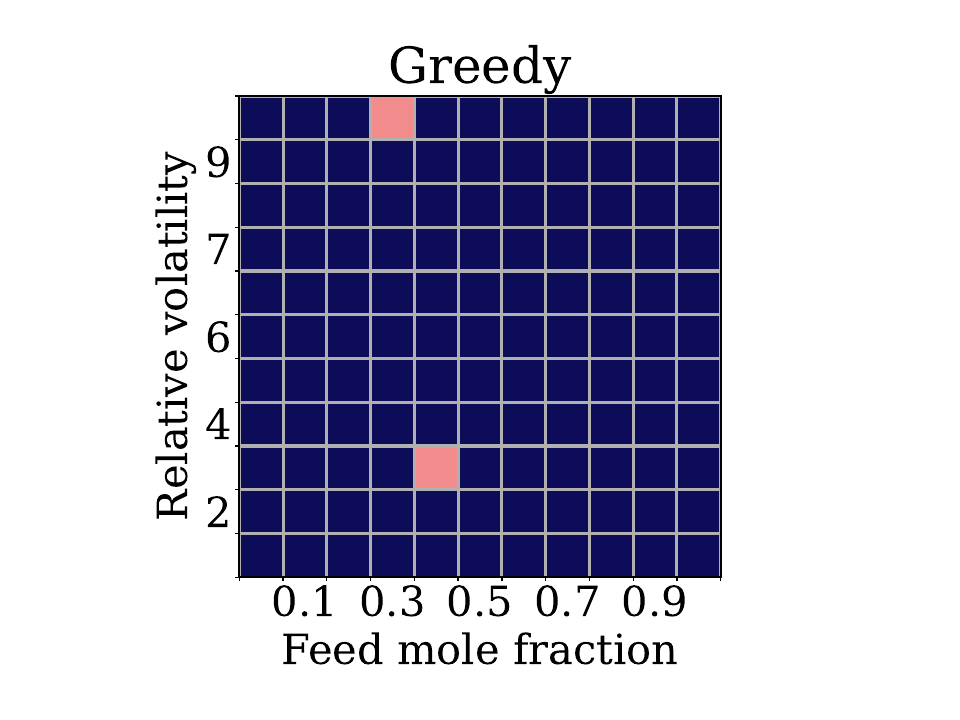}
  \hspace{-0.15\textwidth}\includegraphics[width=0.38\textwidth]{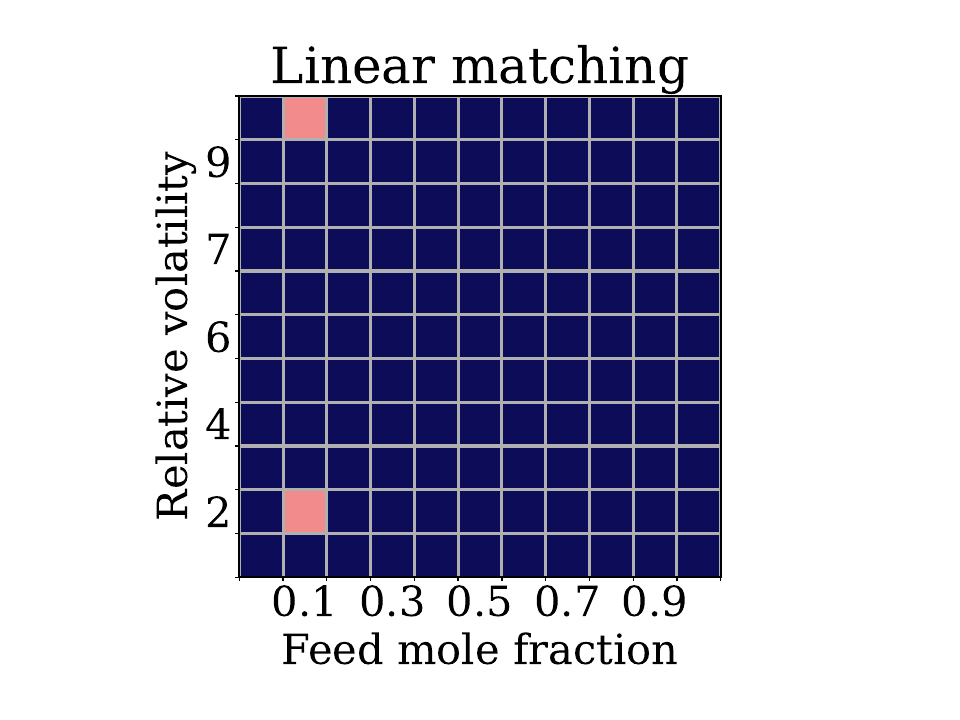}
  }
  \caption{Convergence of the distillation column optimal control problem after
  performing different aggregation strategies}
  \label{fig:distill-convergence}
\end{figure}

\begin{figure}[h]
  \centering
  \resizebox{\textwidth}{!}{
                 \includegraphics[width=5cm]{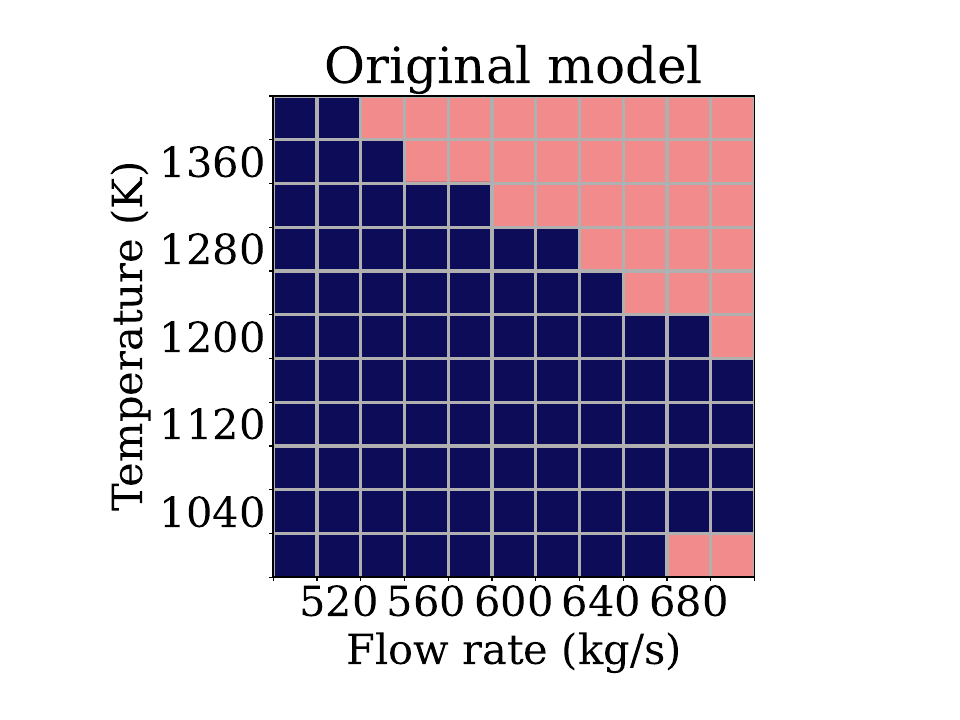}
  \hspace{-0.9cm}\includegraphics[width=6cm]{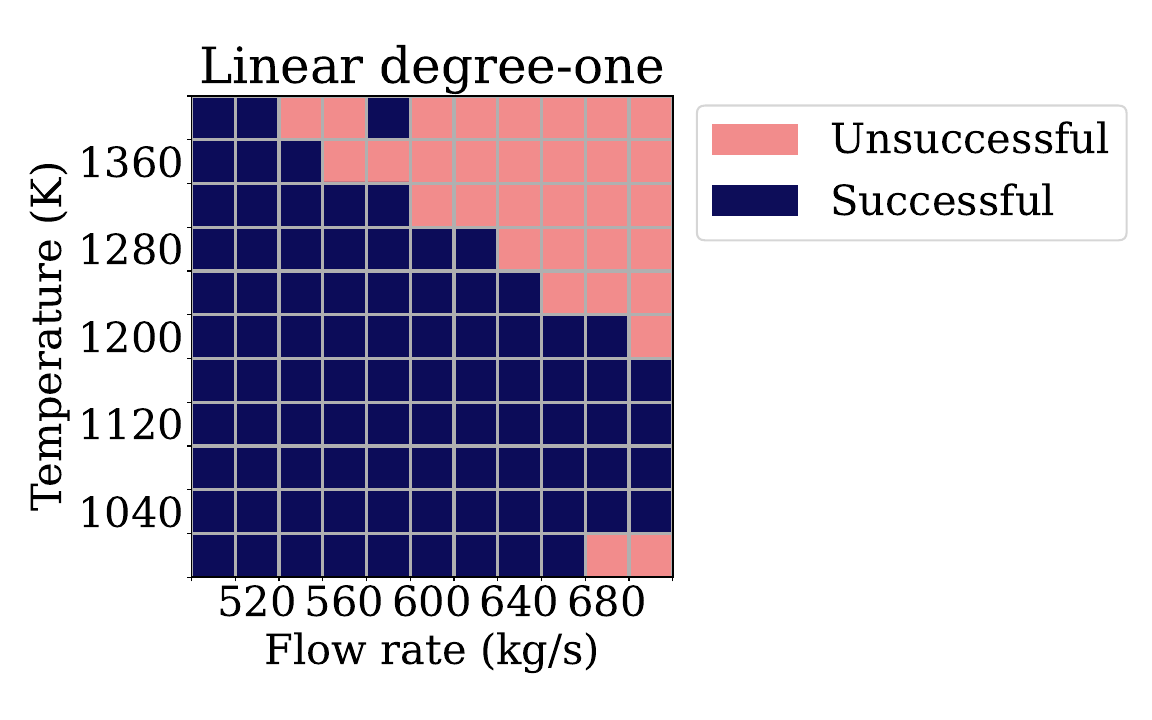}
  \hspace{-0.7cm}\includegraphics[width=5cm]{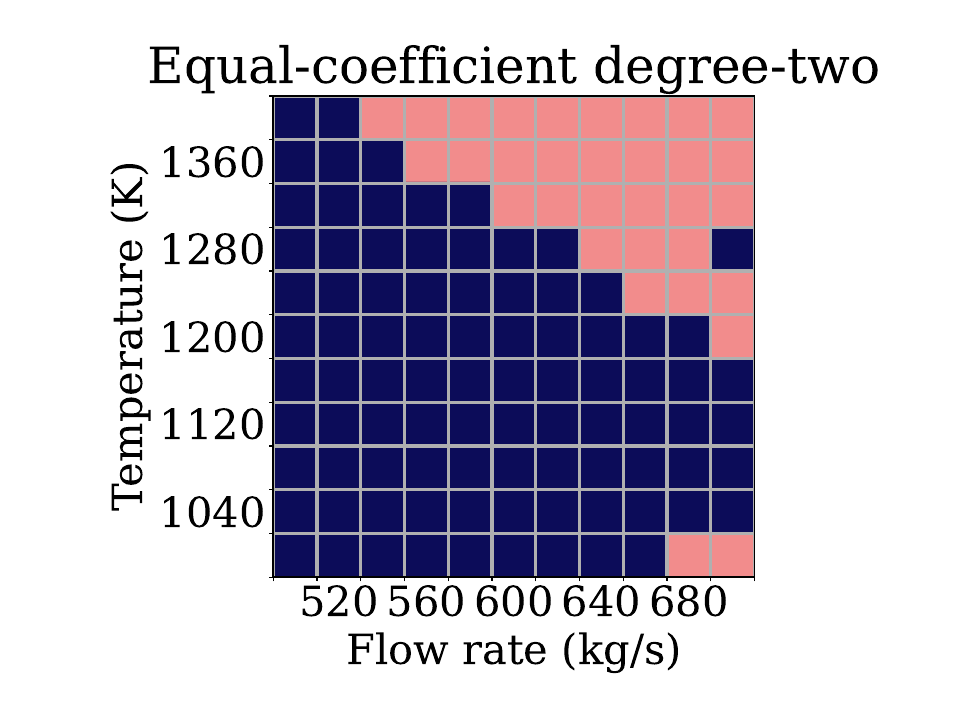}
               }
               \resizebox{1.08\textwidth}{!}{
  \hspace{-0.05\textwidth}\includegraphics[width=5cm]{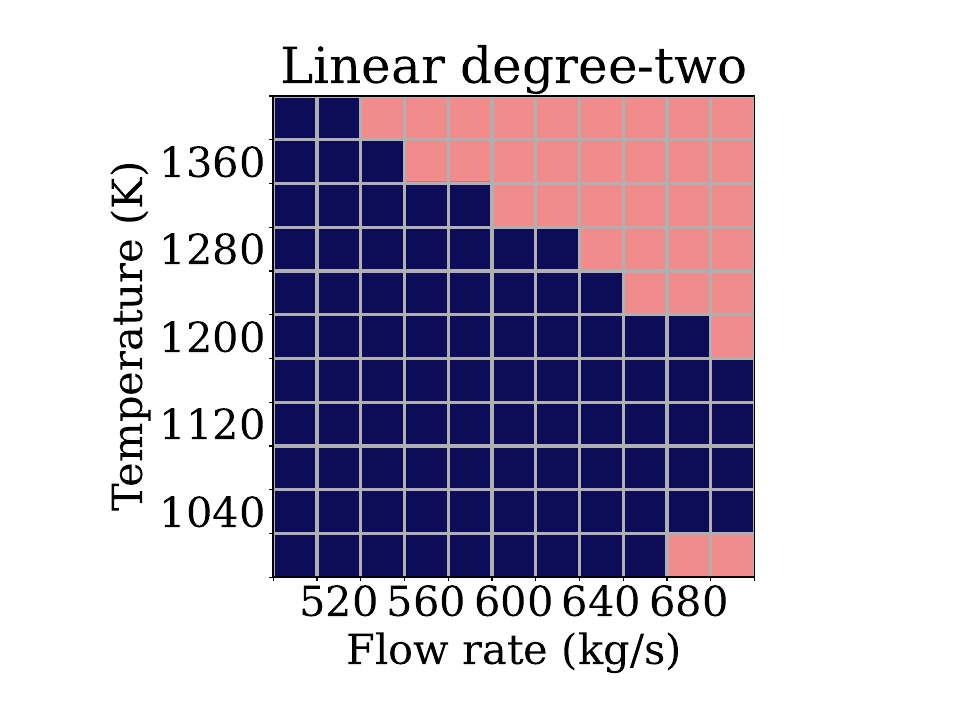}
  \hspace{-1.5cm}\includegraphics[width=5cm]{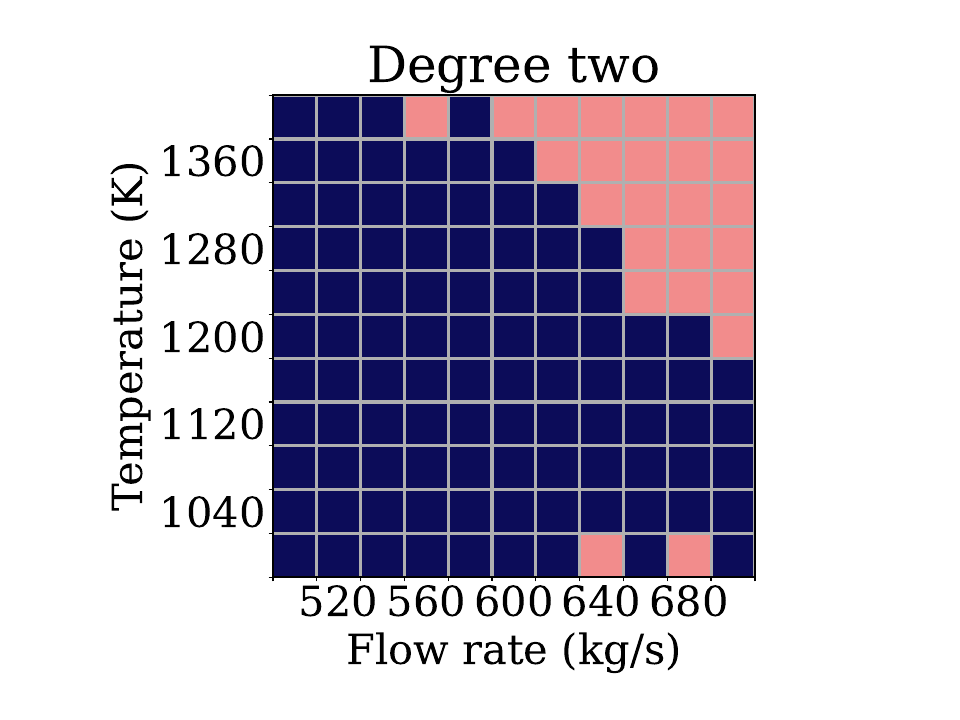}
  \hspace{-1.5cm}\includegraphics[width=5cm]{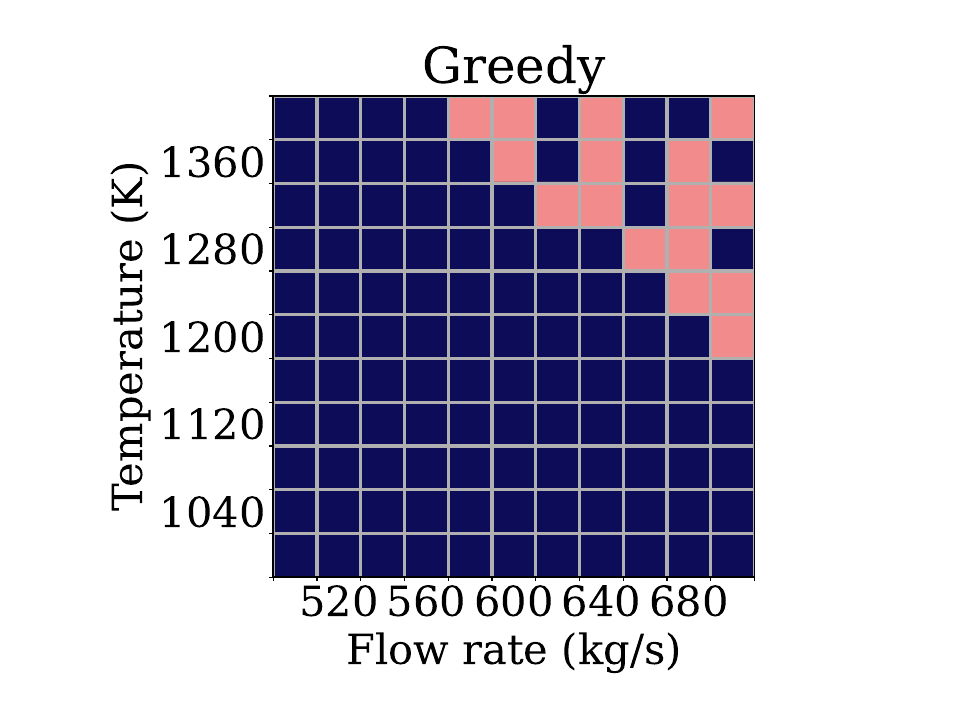}
  \hspace{-1.5cm}\includegraphics[width=5cm]{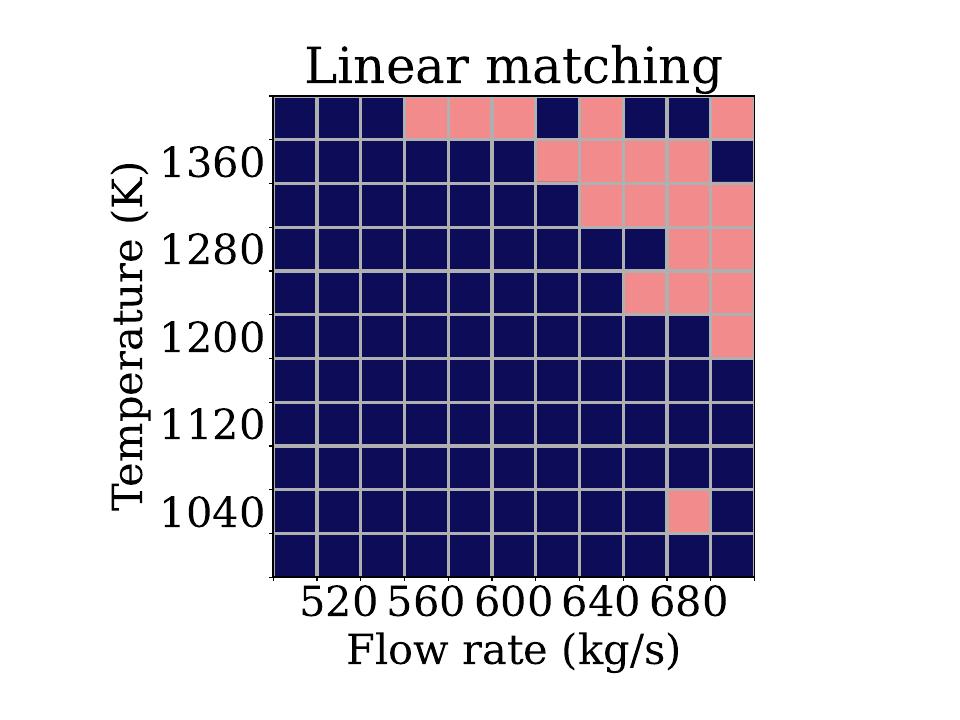}
               }
  \caption{Convergence of the moving bed reactor optimization problem for
  different aggregation strategies}
  \label{fig:mb-steady-convergence}
\end{figure}

\begin{figure}[h]
  \centering
  \resizebox{\textwidth}{!}{
                 \includegraphics[width=5cm]{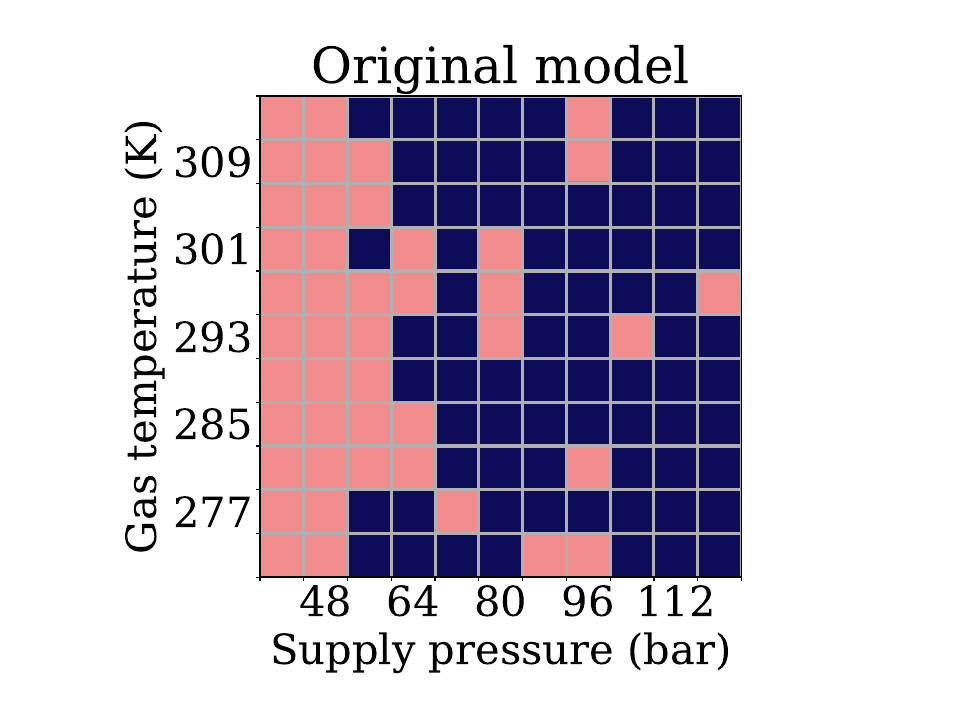}
  \hspace{-0.9cm}\includegraphics[width=6cm]{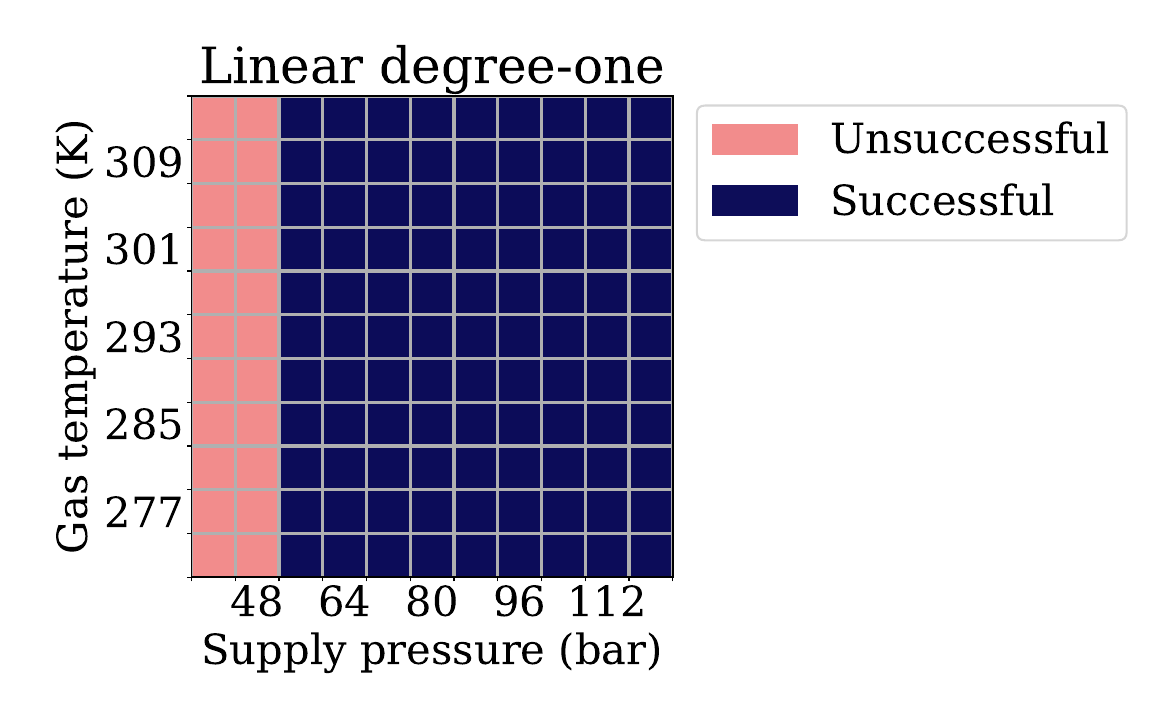}
  \hspace{-0.7cm}\includegraphics[width=5cm]{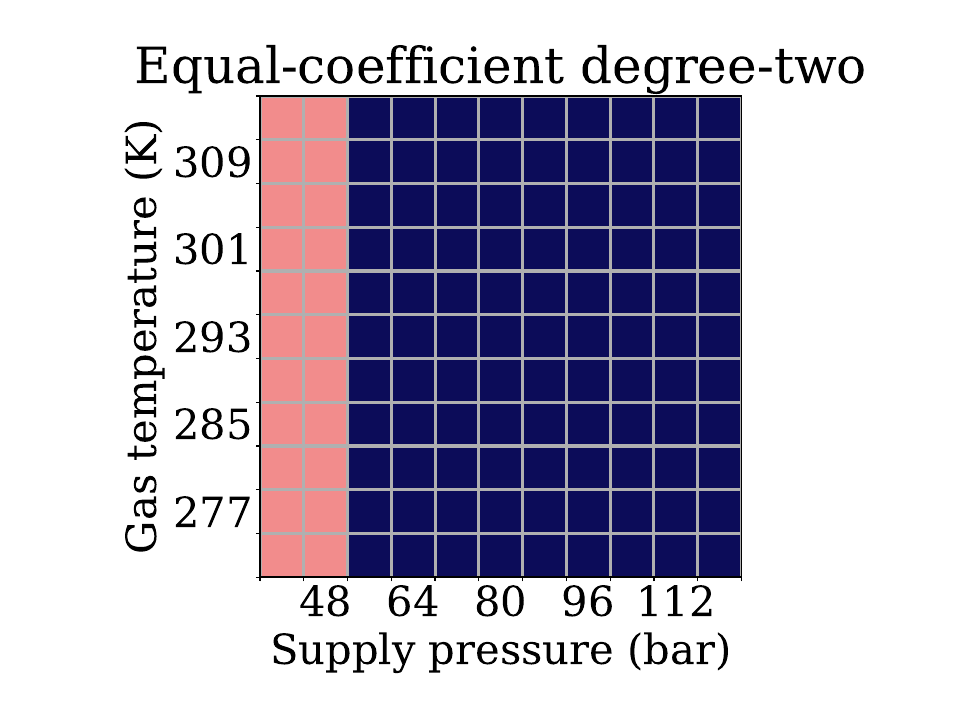}
               }
               \resizebox{1.08\textwidth}{!}{
  \hspace{-0.05\textwidth}\includegraphics[width=5cm]{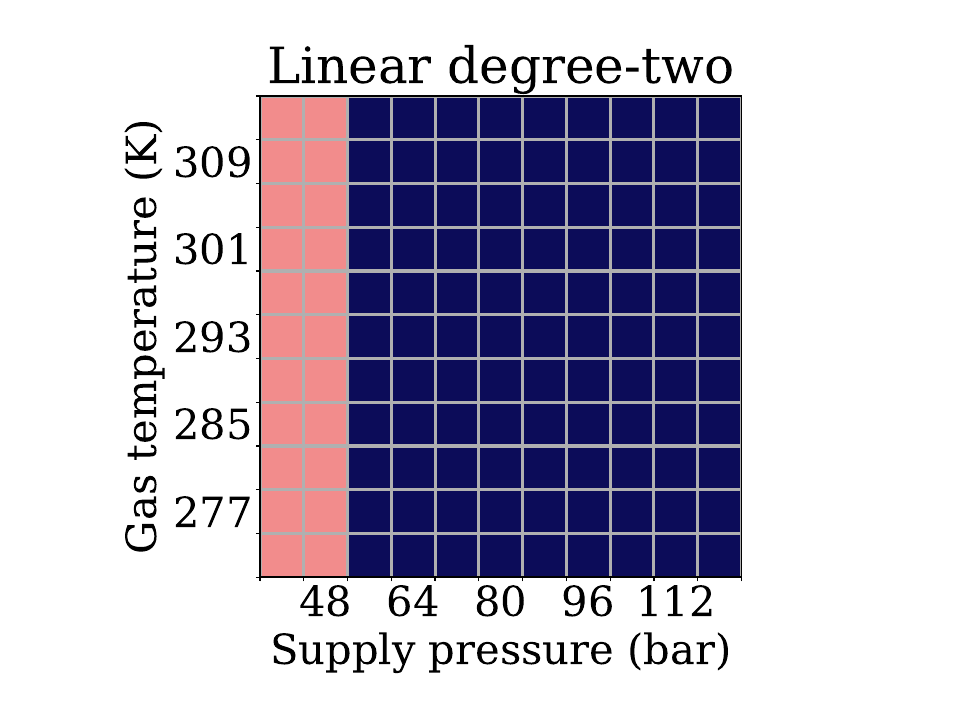}
  \hspace{-1.5cm}\includegraphics[width=5cm]{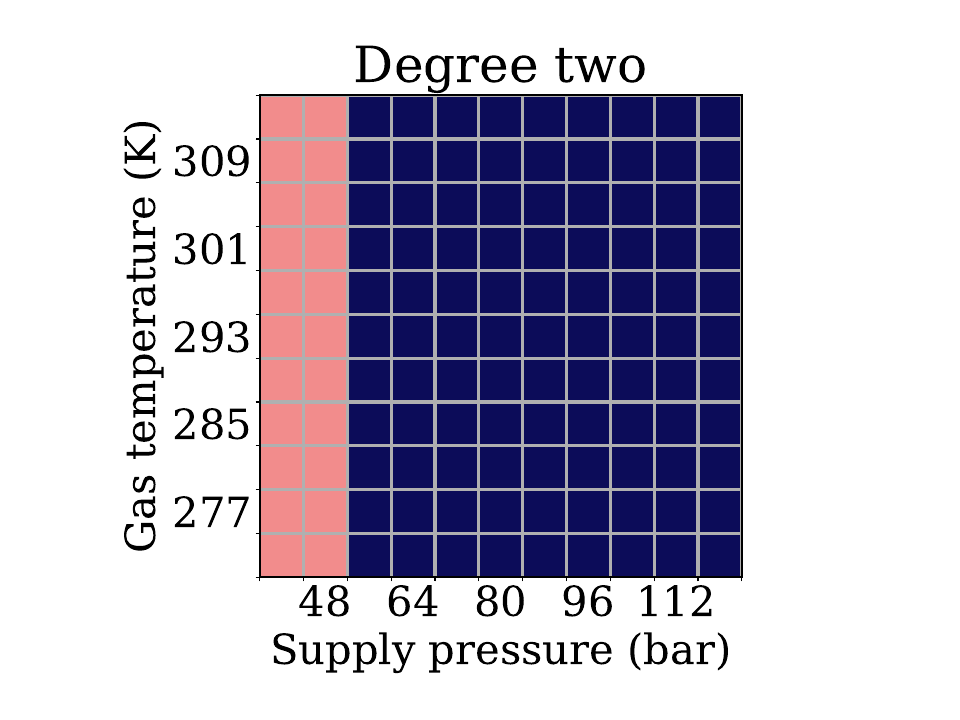}
  \hspace{-1.5cm}\includegraphics[width=5cm]{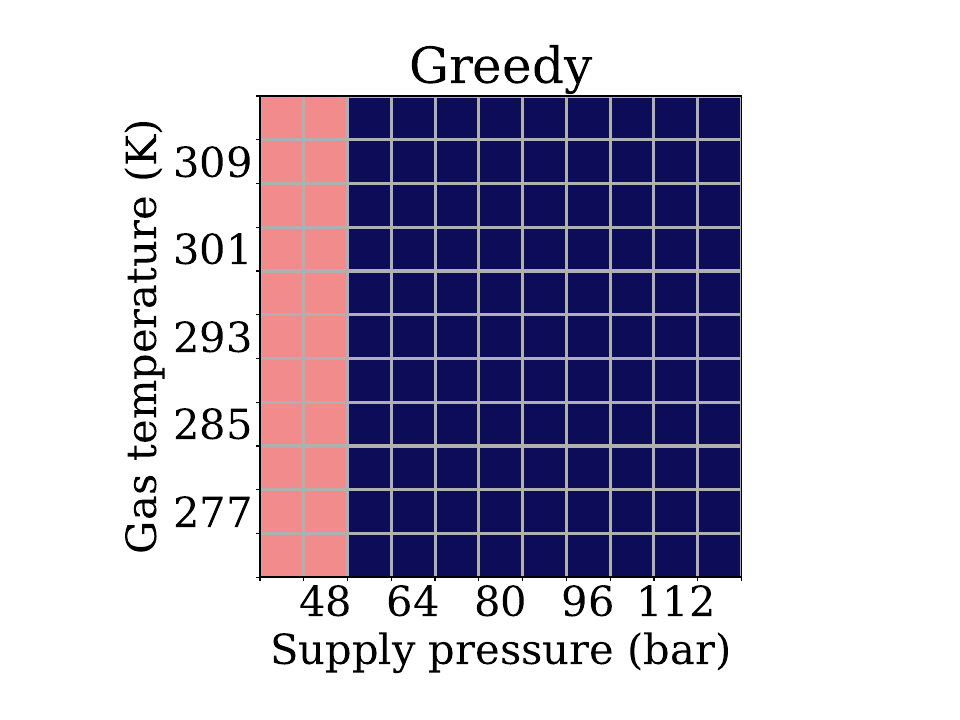}
  \hspace{-1.5cm}\includegraphics[width=5cm]{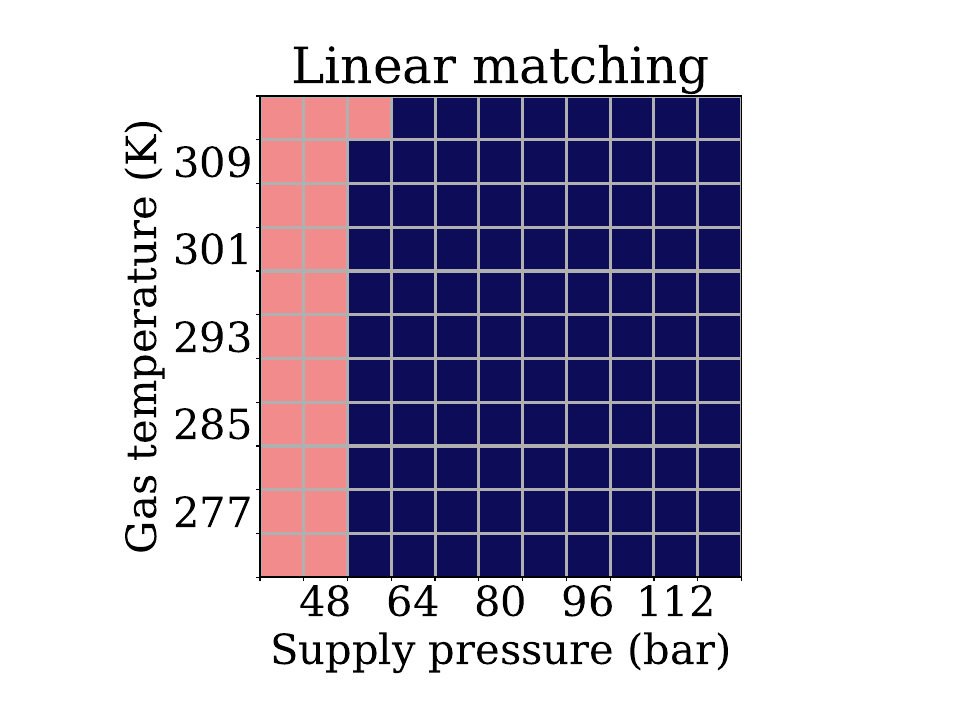}
               }
  \caption{Convergence of the gas pipeline optimization problem for
  different aggregation strategies}
  \label{fig:pipeline_convergence}
\end{figure}

\begin{table}
  \centering
  \caption{Percent of instances converged for methods applied to each model}
  \begin{tabular}{ccccccccc}
    \toprule
    \multirow{2}{*}{Method} & \multicolumn{2}{c}{Distillation} & \multicolumn{2}{c}{Moving bed} & \multicolumn{2}{c}{Pipeline} & Avg \% convergence \\ 
     & \% conv & $t_{solve}^{avg}$
     & \% conv & $t_{solve}^{avg}$
     & \% conv & $t_{solve}^{avg}$ \\ 
\midrule
--        &  85 & 73.2 &  73 & 0.7 & 64 & 49.0 & 74\\
LD1       &  86 & 74.1 &  74 & 0.9 &  82 & 10.4 &  80\\
ECD2      &  88 & 45.9 &  74 & 0.7 &  82 & 6.9 & 81\\
LD2       &  90 & 56.1 &  73 & 0.9 &  82 & 6.5 & 82\\
D2        & 100 & 91.7 &  79 & 0.7 &  82 & 6.1 &  87\\
GR        &  98 & 78.5 & 87 & 0.8 & 82 & 6.3 &  89\\
LM        &  98 & 88.1 & 83 & 0.7 & 81 & 15.9 & 88\\
\bottomrule
  \end{tabular}
  \label{tab:convergence}
\end{table}

Our first observation is that variable aggregation generally improves convergence.
No aggregation method leads to fewer instances converged, for any test problem, than
the original model (represented by ``--'' in the ``Model'' column of Table
\ref{tab:convergence}). For each model, the discrepancy between instances converged
with the original model and with the best aggregation method is significant:
With the distillation model and the degree-2 strategy, 18 additional instances (15\% more) converge,
with the moving bed model and the greedy strategy, 17 additional instances (14\% more) converge,
and with the pipeline model and several different aggregation strategies, 22 additional instances
(18\% more) converge.
Overall, the best strategies are degree-2 (Algorithm \ref{alg:nonlinear-degree2}),
greedy (Algorithm \ref{alg:greedy}), and linear-matching (Algorithm \ref{alg:matching}),
which converge between 13\% and 15\% more instances of these three test problems than
the original model.

Here, methods that aggregate more variables generally converge more often, but the trend
is not monotonic as the greedy strategy outperforms the linear-matching strategy despite
aggregating significantly fewer variables for most models.
Despite some methods performing significantly better than others on average, the
improvement is not uniform for each problem instance. That is, there are some
instances that fail with a ``better'' method despite succeeding with a ``worse'' method.
We note that the reason for success or failure of an interior point method with variables
aggregated or disaggregated is not easy to determine
and believe that investigating the contribution of variable aggregation to enlarging
or shrinking the basin of attraction of the interior point solution is an open and
interesting area for future research.
We believe that some of the improvement is due to primal iterates that stay closer
to the feasible set of the original model and are more likely to have well-conditioned
constraint Jacobians and KKT matrices than points along the infeasible path taken
by an interior point method when solving the original model.
We note that this explanation is similar to the motivation for IPOPT's feasibility
restoration phase (see Section 3.1 of \cite{ipopt}).
\begin{figure}[h]
  \centering
  \resizebox{1.03\textwidth}{!}{
    \hspace{-0.08\textwidth}\includegraphics[width=5cm]{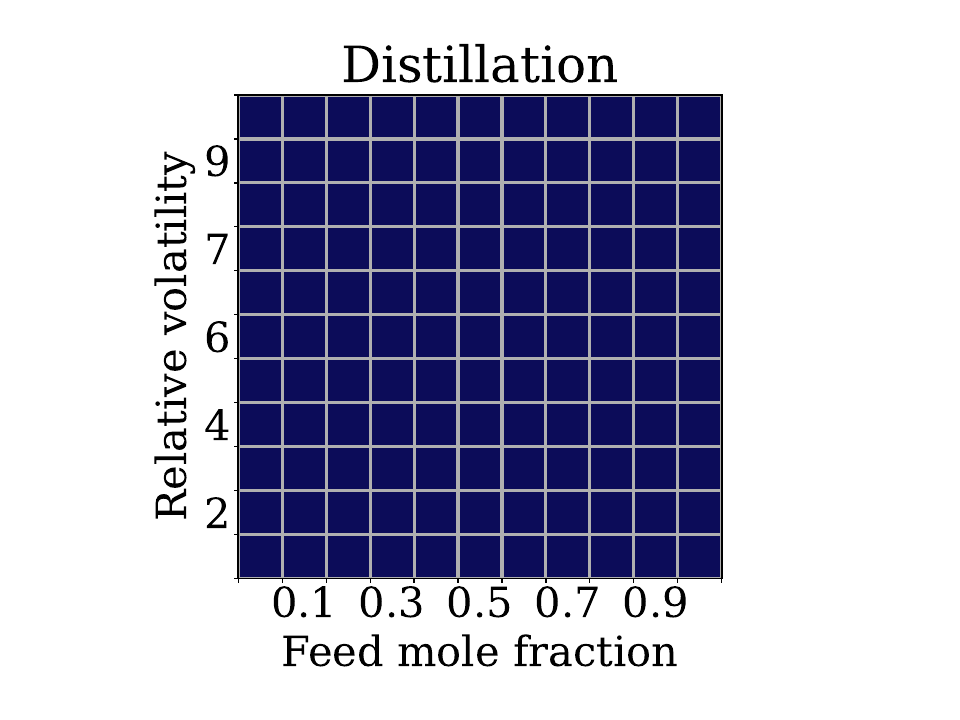}
  \hspace{-1.5cm}\includegraphics[width=5cm]{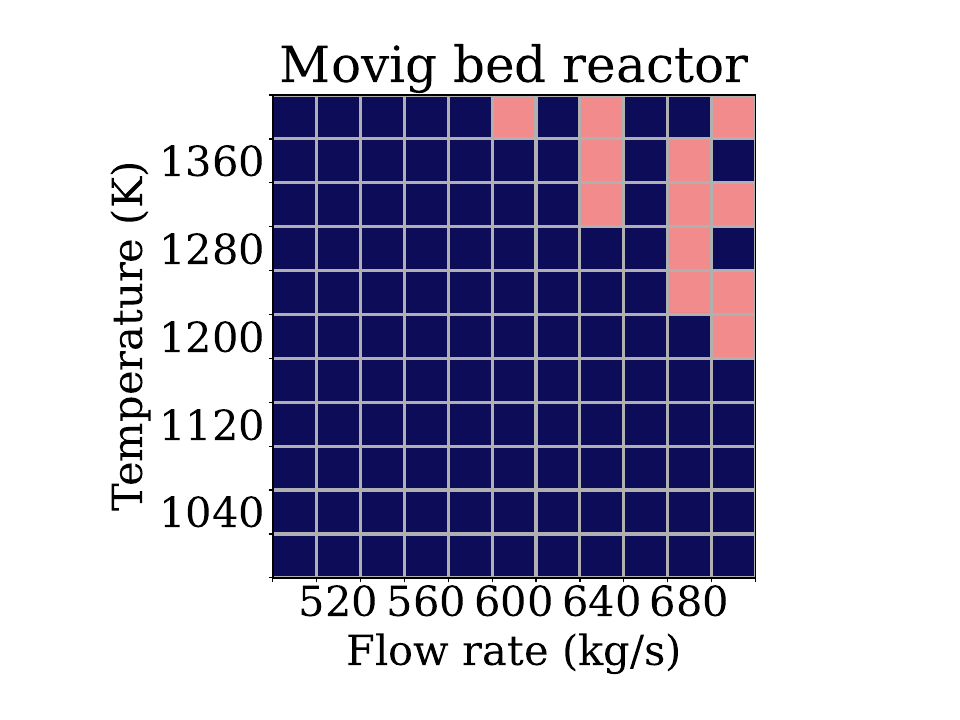}
  \hspace{-1.0cm}\includegraphics[width=6cm]{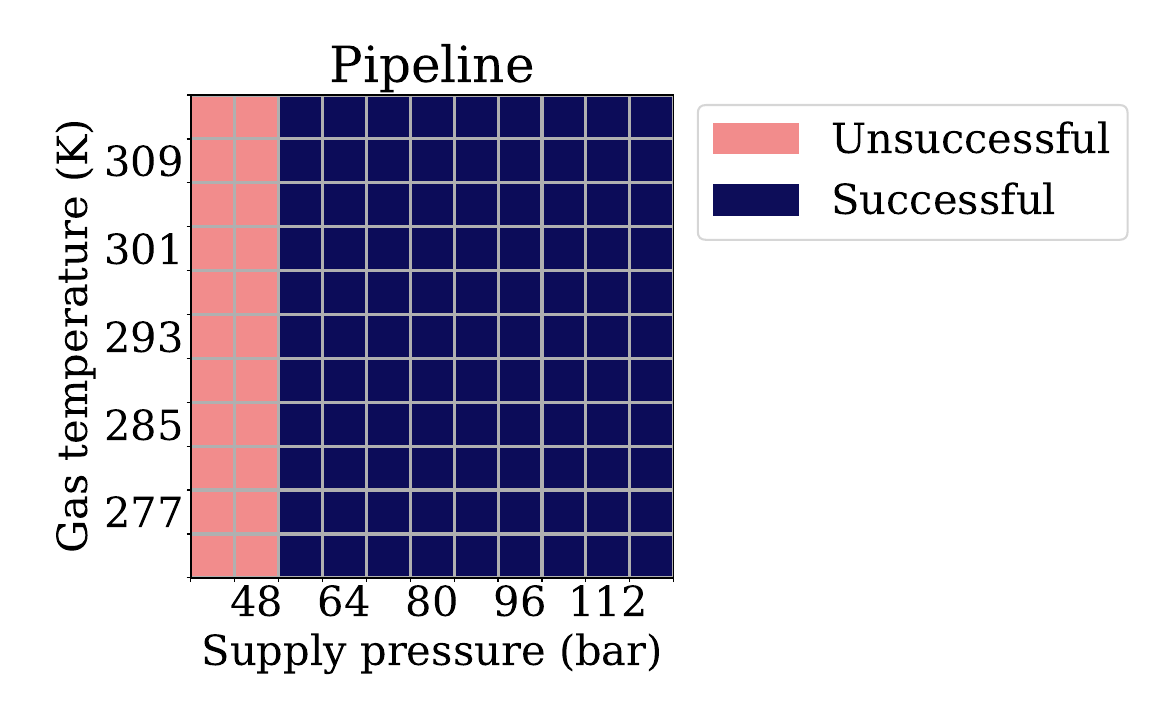}
}
  \caption{Virtual best for all aggregation strategies for distillation, gas pipeline and moving bed reactor model}
  \label{fig:virtual-best}
\end{figure}

In algorithm design, the ``virtual best'' performance is defined as the performance
of a theoretically perfect algorithm for variable aggregation. In this case, the virtual
best algorithm picks the aggregation strategy from among the ones mentioned in Section
\ref{sec:methods} to maximize the convergence across the parameter sweep for each
model. Figure \ref{fig:virtual-best} represents the virtual best for distillation,
gas pipeline and the moving bed reactor models. For distillation, the degree-2 aggregation
strategy is the virtual best since it achieves the same convergence. For the pipeline
problem, all aggregation strategies except linear matching are virtual best. However,
for the moving bed reactor, the virtual best achieves $90 \%$ convergence while greedy---the
best performing aggregation strategy---only achieves $87\%$ convergence.
This indicates that there is further scope for improvement in the variable aggregation
strategies to improve convergence reliability.
\section{Conclusion}
\rev{A novel linear-matching variable aggregation strategy is introduced
  and compared with existing strategies across four
nonlinear optimization problems.}
The strategies strategies are compared across four test problems
in terms of structure, solve time and convergence reliability. The
aggregation strategies are grouped into two main categories:
\begin{itemize}
    \item Structure-preserving strategies that aggregate variables without increasing the density of the remaining constraints
    \item Approximate maximum aggregation strategies that maximize the variables aggregated but may lead to significant increase in the number of non-zeros in the remaining constraints
\end{itemize}
Different aggregation strategies aggregate significantly different number of
variables across all test problems. The approximate maximum aggregation
strategies, specifically the linear matching-based heuristic, eliminate a
significantly higher percentage $(70-90\%)$ of variables compared to the
structure-preserving strategies which eliminate $<60\%$ variables.
However, it
is important to note that the number of non-zeros per constraint stays
approximately the same when structure-preserving aggregation strategies are
used, while the approximate maximum aggregation strategies lead to an increase
in the total number of non-zeros per constraint and also the increase the
nonlinear non-zeros making the problem more nonlinear. 
\rev{We note that the success of the linear-matching strategy at eliminating
  large numbers of variables can also be leveraged to identify large nonsingular
  submatrices for Schur complement decompositions in linear systems,
  which is a potential avenue for future work.}

Variable aggregation can lead to a decrease in solve time as it decreases the factorization time of the KKT matrix. However, there are no guarantees that variable aggregation will decrease solve time as demonstrated in the OPF case where solve time increases after aggregating more variables due to an increase in inequality constraints introduced after elimination of bounded variables. Our work also suggests that Hessian evaluation can become a bottleneck if more variables are aggregated without preserving the structure of the problem, and hence, structure-preserving aggregation strategies are valuable.  

The convergence reliability experiments indicate that variable aggregation leads to more
reliable problem formulations. For the distillation example, degree-2 aggregation
strategy leads to $100 \%$ convergence over the parameter sweep.
One reason for improved convergence may be that the aggregated problems follow a ``more
feasible'' path with respect to the constraints of the original problem,
leading to a better-conditioned KKT matrix.
It is also observed that degree-2, greedy, and linear-matching strategies lead
to more reliable convergence than the linearity-preserving, structure-preserving
strategies.
This result reveals an important trade-off: While aggregating many variables may
introduce a new bottleneck due to expensive Hessian evaluation, it may lead to
significantly more reliable convergence for some problems. For this reason,
we recommend the degree-2 aggregation method, which strikes a balance between
preserving structure and aggregating many variables, as a general pre-solve reduction.

While we observe that aggregation generally leads to more reliable convergence, we note
that (a) the improvement is not uniform as more variables are eliminated and (b) there
may be other problems for which aggregation causes worse convergence reliability.
For instance, Albersmeyer and Diehl \cite{albersmeyer2010lifted} suggest ``lifting''
a problem into a higher-dimensional space to improve iteration counts.
Explaining the effects of problem reformulations on algorithms' convergence properties
is a significant challenge for nonlinear optimization, and this work will serve as a basis
for further comparison and explanation of convergence behavior.

Based on our results, we believe that approximate-maximum and incidence-preserving
aggregation methods should be implemented as pre-solve options in nonlinear modeling
environments and solvers \rev{with an incidence-preserving aggregation method as the default option}.
Additionally, we believe that the following efforts should
be undertaken to explain and confirm the trends we observe:
\begin{enumerate}
  \item Theoretical analysis of interactions among solvers' \rev{global convergence} methods
    \rev{(e.g., line search and trust region methods)} and
    different aggregation strategies.
  \item Development of a test set of challenging, parameterized nonlinear optimization
    problems for which convergence of state-of-the-art solvers is not 100\% reliable.
\end{enumerate}
With these efforts, the methods we present may be further analyzed, theoretically and
experimentally, and more informed development of nonlinear optimization software
may proceed.

\bibliography{ref}

\section*{Statements and Declarations}
\rev{We gratefully acknowledge funding from the Los Alamos National Laboratory
Laboratory-Directed Research \& Development program through the Center for Nonlinear Studies
(CNLS).
LA-UR-25-21592.

\noindent We have no competing interests to disclose.}

\appendix
\section{Block triangularization}
\label{sec:block-triang}

The block triangularization algorithm takes an undirected bipartite graph $G=(A,B,E)$
and a perfect matching thereof $\mathcal{M}$ as inputs and returns an ordered partition
of $\mathcal{M}$, $\mathcal{B}$, that defines the irreducible block lower triangular
form.
Algorithms for permuting a sparse matrix or bipartite graph to irreducible block
triangular form are well described by Duff and Reid
\cite{duff1978implementation,duff1978algorithm} and Pothen and Fan \cite{pothen1990}.
These descriptions define the algorithm as returning partitions of node sets $A$ and
$B$, or of rows and columns of the sparse matrix.
However, to facilitate the presentation of the matching-based aggregation algorithm,
Algorithm \ref{alg:matching}, we define
the block triangularization algorithm to return a partition of edges. For completeness,
we provide a full description of the block triangularization algorithm in this appendix.
The algorithm is defined using operations on a directed graph.

A {\it directed graph} $G=(V, E)$ is a graph in which edges are {\it ordered} pairs
of nodes. If $E$ contains an edge $(u, v)$, $v$ is said to be {\it out-adjacent} to $u$,
while $u$ is {\it in-adjacent} to $v$.
In a directed graph, a {\it path} is a sequence of nodes $\{u_1, \dots, u_n\}$ where,
for every adjacent pair in the sequence $(u_i, u_{i+1})$, $u_{i+1}$ is out-adjacent
to $u_i$. A {\it strongly connected component} $C$ is a subset of nodes in a graph
where, for every pair of nodes $u,v\in C$, there exists a path from $u$ to $v$.
The set of all strongly connected components $\mathcal{C}$ partitions the nodes in a
directed graph.
Strongly connected components of a directed graph may be computed in
$\mathcal{O}(n_v+n_e)$ time, where $n_v$ is the number of vertices and $n_e$ is the
number of edges, by Tarjan's algorithm \cite{tarjan1972}.
A {\it cycle} is a path that starts and ends with the same node, and a {\it directed acyclic
graph} (DAG) is a directed graph that has no cycles.
The strongly connected components of a directed graph form a DAG that can be constructed
by the {\tt compress} subroutine described in Algorithm \ref{alg:compress}.
A {\it topological order} is a (nonunique) permutation of the nodes in a DAG such that,
if $(u, v)$ is an edge in the DAG, $u$ comes before $v$ in the topological order.
A topological order may be computed in $\mathcal{O}(n_e+n_v)$ time \cite{manber1989}.

The block triangularization algorithm is given by Algorithm
\ref{alg:block-triangularize}. It consists of five major steps:
\begin{enumerate}
  \item Project the bipartite graph into a directed graph defined over the
    bipartite set $A$, using the matching $\mathcal{M}$ to determine edge orientation.
  \item Partition the nodes $A$ into strongly connected components of the new
    directed graph.
  \item Compress the strongly connected components into a DAG. Each node of the
    DAG is a subset of nodes in the directed graph (a subset of $A$).
  \item Compute a topological order of the DAG.
  \item For each subset of $A$ in the topological order, convert each node $a$
    to an edge $(a,b)$ by finding the matched vertex in $\mathcal{M}$.
\end{enumerate}
Descriptions of the projection and compression subroutines are given in
Algorithms \ref{alg:project} and \ref{alg:compress}, and Table
\ref{tab:bt-subroutines} summarizes all subroutines used by the block
triangularization algorithm.
By construction, the returned set of sets of edges, $\mathcal{T}$
in Algorithm \ref{alg:block-triangularize}, is a partition of the perfect
matching $\mathcal{M}$.

\begin{table}
  \centering
  \caption{Subroutines used by {\tt block\_triangularize} algorithm}
  \resizebox{\textwidth}{!}{
    \begin{tabular}{cp{0.2\linewidth}p{0.2\linewidth}cp{0.3\linewidth}}
      \toprule
    Subroutine & Inputs & Outputs & Time complexity & Description \\
    \midrule
    {\tt adjacent\_to} & Graph $G$, node $u$ & Set of nodes & $\mathcal{O}(\delta(u))$ & Set of nodes in $G$ adjacent to $u$ \\
    {\tt matched\_with} & Matching $\mathcal{M}$, node $u$ & Node & $\mathcal{O}(1)$ & Node matched with $u$ in $\mathcal{M}$\\
    {\tt subset\_containing} & Set of disjoint sets, node $u$ & Set & $\mathcal{O}(1)$ & Return set containing $u$ \\
    \midrule
    {\tt project} & Bipartite graph, matching & Directed graph & $\mathcal{O}(n_e)$ & Directed graph defined on one bipartite set\\
    {\tt strongly\_connected\_comps} & Directed graph & Set of sets of nodes & $\mathcal{O}(n_v + n_e)$ & Return the strongly connected components of the graph \\
    {\tt compress} & Graph, set of sets of nodes & Graph & $\mathcal{O}(n_v + n_e)$ & Each specified subset of nodes in the original graph is a single node in the returned graph \\
    {\tt topological\_sort} & Directed acyclic graph & Ordered set of nodes & $\mathcal{O}(n_v+n_e)$ & Permute the nodes into topological order \\
    \bottomrule
  \end{tabular}
  }
  \label{tab:bt-subroutines}
\end{table}

\begin{algorithm}
  \caption{: {\tt block\_triangularize}}
  \begin{algorithmic}[1]
    \State {\bf Inputs:} Bipartite graph $G=(A,B,E)$, perfect matching $\mathcal{M}$
    \State $G_d = \text{\tt project}(G, \mathcal{M})$ \Comment{The nodes of $G_d$ are $A$}
    \State $\mathcal{C} = \text{\tt strongly\_connected\_comps}(G_d)$ \Comment{$\mathcal{C}$ partitions $A$}
    \State $D = \text{\tt compress}(G_d, \mathcal{C})$ \Comment{The nodes of $D$ are $\mathcal{C}$}
    \State $\mathcal{S} = \text{\tt topological\_sort}(D)$
    \State $\mathcal{T} = \{\}$
    \For{$C\text{~\bf in }\mathcal{S}$} \Comment{$C$ is a subset of $A$}
    \State $T = \{\}$ \Comment{$T$ is a subset of  $\mathcal{M}$}
      \For{a$\text{~\bf in } C$}
        \State $b=\text{\tt matched\_with}(\mathcal{M}, a)$
        \State $T \gets T \cup (a, b)$
      \EndFor
      \State $\mathcal{T} \gets \mathcal{T} \cup T$
    \EndFor
    \State {\bf Return:} $\mathcal{T}$
  \end{algorithmic}
  \label{alg:block-triangularize}
\end{algorithm}

\begin{algorithm}
  \caption{: {\tt project}}
  \begin{algorithmic}[1]
    \State {\bf Inputs:} Bipartite graph $G=(A,B,E)$, matching $\mathcal{M}$
    \State $E_d=\{\}$
    \For{$a\text{~\bf in } A$}
      \State $b={\tt matched\_with}(\mathcal{M}, a)$
      \For{$\bar a\text{~\bf in }\text{\tt adjacent\_to}(b)$}
        \State $E_d \gets E_d \cup (\bar a, a)$
      \EndFor
    \EndFor
    \State {\bf Return:} $G_d=(A,E_d)$
  \end{algorithmic}
  \label{alg:project}
\end{algorithm}

\begin{algorithm}
  \caption{: {\tt compress}}
  \begin{algorithmic}[1]
    \State {\bf Inputs:} Graph $G=(V,E)$, partition $\mathcal{C}$
    \State $V_\mathcal{C} =\mathcal{C}$
    \State $E_\mathcal{C} = \{\}$
    \For{$C\text{~\bf in } \mathcal{C}$}
      \For{$u\text{~\bf in } C$}
        \For{$v\text{~\bf in } \text{\tt adjacent\_to}(u)$}
          \State $\bar C = \text{\tt subset\_containing}(\mathcal{C}, v)$
          \If{$C \neq \bar{C}$}
            \State $E_\mathcal{C}\gets E_\mathcal{C} \cup (C, \bar C)$
          \EndIf
        \EndFor
      \EndFor
    \EndFor
    \State {\bf Return:} $G_\mathcal{C}=(V_\mathcal{C},E_\mathcal{C})$
  \end{algorithmic}
  \label{alg:compress}
\end{algorithm}

\end{document}